\documentclass[reqno]{amsart}

\usepackage{amsmath,amssymb,amsthm,amsfonts,mathtools}
\usepackage{hyperref}
\usepackage[english]{babel}
\usepackage{cancel}
\usepackage{xcolor}
\usepackage{csquotes}
\usepackage{inputenc}
\usepackage{fontenc}
\usepackage{tikz}

\newcommand{\cA}{{\mathcal A}}
\newcommand{\cB}{{\mathcal B}}
\newcommand{\cD}{{\mathcal D}}
\newcommand{\cF}{{\mathcal F}}

\newcommand{\cK}{{\mathcal K}}
\newcommand{\cL}{{\mathcal L}}
\newcommand{\cM}{{\mathcal M}}

\newcommand{\cQ}{{\mathcal Q}}

\newcommand{\cU}{{\mathcal U}}

\newcommand{\myspace}{\qquad\qquad\qquad}

\newtheorem{theorem}{Theorem}[section]
\newtheorem{lemma}[theorem]{Lemma}
\newtheorem{proposition}[theorem]{Proposition}
\newtheorem{remark}[theorem]{Remark}
\newtheorem{remarks}[theorem]{Remarks}

\newtheorem{assumptions}[theorem]{Assumptions}

\newtheorem{definition}[theorem]{Definition}

\newtheorem{problem}[theorem]{Problem}
\numberwithin{equation}{section}

\date{}

\begin{document}

\title[]{Riccati-based solution to the optimal control
of linear evolution equations with finite memory}

\author{Paolo Acquistapace}
\address{Paolo Acquistapace, Universit\`a di Pisa, Dipartimento di Matematica ({\em retired}),
Largo Bruno~Pontecorvo 5, 56127 Pisa, ITALY 
}
\email{paolo.acquistapace(at)unipi.it}

\author{Francesca Bucci}
\address{Francesca Bucci, Universit\`a degli Studi di Firenze,
Dipartimento di Matematica e Informatica,
Via S.~Marta 3, 50139 Firenze, ITALY
}
\email{francesca.bucci(at)unifi.it}

\makeatletter
\@namedef{subjclassname@2020}{\textup{2020} Mathematics Subject Classification}
\makeatother



\begin{abstract}
In this article we study the optimal control problem with quadratic functionals for a linear Volterra integro-differential equation in Hilbert spaces.
With the finite history seen as an (additional) initial datum for the evolution, following the variational approach utilized in the study of the linear-quadratic problem for memoryless infinite dimensional systems, we attain a closed-loop form of the unique optimal control via certain operators that are shown to solve a coupled system of  
quadratic differential equations.
This result provides a first extension to the partial differential equations realm 
of the Riccati-based theory recently devised by L.\,Pandolfi in a finite dimensional context.
\end{abstract}

\maketitle


\section{Introduction} \label{s:intro}
Given a linear control system $y'(t)=Ay(t)+Bu(t)$, $t\in [0,T)$, in a Hilbert space $H$,
existence and uniqueness for the Cauchy problems associated with the (differential) Riccati equations 
\begin{equation*}
P'(t)+A^*P(t)+P(t)A-P(t)BB^*P(t)+Q=0 \qquad t\in [0,T)
\end{equation*}
plays a central role in the study of the associated optimal control problems with quadratic functionals on a finite time interval.~(We recall that the operator $Q$ occurs in the functional
in connection with the observed state.)
Indeed, the Riccati equation corresponding to the minimization problem is expected to yield the -- hopefully unique -- operator $P=P(t)$ which enters the feedback representation of the optimal control, thereby allowing its synthesis.

The question of well-posedness of Riccati equations is particularly interesting in the context of partial differential equations (PDE), where major technical difficulties stem from the presence of unbounded control operators $B$ -- these are brought about, in particular, by the modeling of boundary inputs.
The analysis become even more challenging in the case of hyperbolic or composite dynamics:
the actual meaning of the gain operator $B^*P(t)$ that occurs in the quadratic term of the Riccati equation (or lack thereof) is the central issue that must be tackled, even for the only purpose
of existence of a solution to the Riccati equation.

Forty years of research on this subject have brought about distinct functional-analytic frameworks that mirror parabolic PDE, hyperbolic PDE, and also certain systems of hyperbolic-parabolic PDE, along with respective Riccati theories. 
The reader is referred to the Lecture Notes \cite{las-trig-lncis} for a brief overview and significant illustrations on the subject, and to the monographs by Bensoussan~{\em et al.} \cite{bddm} and by Lasiecka and Triggiani \cite{las-trig-redbooks} for an ample treatment of the Riccati theories pertaining to the finite time horizon problem until 2000.
A literature review on the major contributions and latest achievements in the study of the linear-quadratic (LQ) problem for coupled systems of hyperbolic-parabolic PDEs -- now spanning more than two decades -- is found in the recent \cite{ac-bu-uniqueness_2023} and \cite{tuffaha-stoch_2023}
(in the deterministc and stochastic cases, respectively).
 
\smallskip 
In this paper we are instead interested in the LQ problem for integro-differential equations.
These arise, as is widely known, in the modeling of certain diffusion processes and other phenomena
that exhibit hereditary effects such as, e.g., viscoelasticity; see for instance the monograph
\cite{renardy-etal_1987} by Renardy~{\em et al.}
The class of evolutions we consider is described by problem \eqref{e:state-eq_0}, which is introduced in full detail in Section \ref{ss:framework} below. 
It is a reasonably simple model equation, which reduces to the archetypical differential system
in the absence of the integral term, with the control operator $B$ here assumed bounded.
We believe the setting and the chosen approach should also serve as a baseline for further developments.

When it comes to the LQ problem for integro-differential PDE,
the picture of the existing literature is not as complete as in the memoryless case.
On one hand, the work by Cannarsa {\em et al.} \cite{cannarsa-etal_2013} -- as first, in 2013 --
deals with the Bolza problem for semilinear evolution equations, with the dynamics displaying an {\em infinite} memory; the analysis encompasses both parabolic and hyperbolic equations with nonlocal terms (the respective results and methods of proof differ\footnote{Just to give a glimpse of some of the methods employed therein, we recall that in order to tackle a second order in time problem the authors appeal to the celebrated {\em history approach} introduced by Dafermos in 1970, which allows to utilize semigroup theory for the mathematical analysis of an equivalent coupled system satisfied by a suitable augmented variable.}, though).
Thus, \cite{cannarsa-etal_2013} deals with a more general problem than the LQ one;
results for the existence of an {\em open-loop} solution to the optimization problem are established, which apply to relevant physical evolutions.

On the other hand, our framework and goals are more specific: we aim at the LQ problem and its {\em closed-loop} optimal solution, and more precisely, at a synthesis of the unique optimal control   
by way of solving a corresponding Riccati equation.
Thus, we must go back in time a bit.
The 1996 work \cite{pritchard-you_1996} by Pritchard and You addresses a similar problem
(more precisely, our integro-differential model can be subsumed under the one considered therein), in the same Hilbert space setting.
A semi-causal representation formula for the optimal control is established, with the feedback operator
depending on an another operator which is shown to solve a Fredholm integral equation.
In the authors' words, the said equation
\textquote{\dots plays a role similar to that of the operator Riccati equation}.

The question as to whether a Riccati-based theory is actually viable was addressed (and answered affirmatively) just recently by L.~Pandolfi \cite{pandolfi-memory_2018}, with a study restricted to an uncomplicated setting, namely, considering the integro-differential model \eqref{e:state-eq_0} in a finite dimensional space $H=\mathbb{R}^n$ and neglecting the generator $A$ (governing the free dynamics in the absence of memory).
A feedback representation of the optimal control is established, and in addition the operators involved in it are shown to solve a coupled system of three quadratic (matrix) equations. 

In this work we fully extend -- and to some extent push forward, see our Theorem~\ref{t:big-riccati} -- the results obtained in \cite{pandolfi-memory_2018} to the controlled integro-differential system \eqref{e:state-eq_0} in a true infinite dimensional context. 
In order to do so,
\begin{itemize}

\item
we adopt the Volterra equations perspective of \cite{pandolfi-memory_2018}, along with the consideration of the history as a component of the state (shared by \cite{cannarsa-etal_2013} as well), while

\item
we perform (and adapt) the plan carried out in the study of the LQ problem for 
{\em memoryless} control systems in infinite dimensional spaces, whose line of argument
can be summarized (in broad terms) as follows:
(i) a convex optimization argument brings about 
(ii) an operator $P(t)$, defined in terms of the optimal evolution;
(iii) $P(t)$ is shown to solve a Riccati equation;
(iv) as the (hopefully, unique) Riccati operator $P(t)$ enters the feedback law,
it renders its synthesis effective.
\end{itemize}
The formulation of the optimal control problem, our assumptions and main results, i.e.
Theorems~\ref{t:main} and \ref{t:big-riccati}, are made explicit in the next Section~\ref{ss:framework}.
An expanded outline of the paper found at the end of this section will provide guide and insight into the sequence of proofs.

\smallskip
In concluding this introduction we provide a 
(minimal, for obvious reasons) bibliographical selection of general textbooks as well as articles with specific focuses, the latter ones still
within control theory for linear models. 

An explicit account of the well-posedness and regularity results devised for abstract linear equations
in Banach spaces 
is beyond the scope of this article.
We point out and give credit to a few pioneering works from the late sixties on, such as
\cite{friedman-shinbrot_1967}, \cite{dafermos_1970} and \cite{daprato-iannelli_1980}.
%
So, with regard to the broad topic of integro-differential equations the reader is referred to the monographs (listed in cronological order, along with the above-mentioned \cite{renardy-etal_1987}) by Gripenberg {\em et al.}~\cite{gripenberg-etal_1990} (in finite dimensional spaces), Pr\"uss \cite{pruess_1993} and Pandolfi \cite{pandolfi-book}; \cite{pandolfi-book} is especially valuable for the 
historical insight and up-to-date references, besides the discussion of modeling and analytical aspects.
%
(In this connection we note that even a higher order PDE such as the Moore-Gibson-Thompson equation\footnote{The Moore-Gibson-Thompson (or Stokes-Moore-Gibson-Thompson) equation is a widely studied third order in time PDE arising from the linearization of a quasilinear model for the propagation of ultrasound waves.} has been related -- with different aims and in a different fashion -- 
to wave equations with memory; see \cite{delloro-pata_2017}, \cite{bucci-pandolfi_2020},
\cite{bucci-eller_2021}.)

Now, moving on to various 
control-theoretic properties -- distinct from quadratic optimal control -- that have been explored and
established 
in the case of evolution equations with (infinite, more often than finite) memory,
we recall that 
these include controllability (\cite{ivanov-pandolfi_2009}, \cite{guerrero-imanuvilov_2013}, \cite{chaves-silva-etal_2014,chaves-silva-etal_2017}, \cite{fernandezcara-etal_2016,fernandezcara-etal_2022}),
reachability (\cite{loreti-sforza_2010}, \cite{gamboa-etal_2016}), 
unique continuation (\cite{doubova-fernandezcara_2012});
observability and inverse problems via Carleman estimates 
(\cite{cavaterra-etal_2006}, \cite{loreti-etal_2017a,loreti-etal_2017b}); 
stability and uniform decay rates
(\cite{grasselli-squassina_2006}, \cite{chepyzhov-pata_2006}, \cite{conti-etal_2008}).
%
The mathematical tools utilized include: purely PDE methods, semigroup theory, harmonic analysis.
(The variety of notable advances in the study of 
the long-time behaviour of solutions to semilinear and nonlinear equations with memory is inevitably left out.) 

We note in particular that a line of investigation which has been followed specifically in the study of controllability is the reduction of the integro-differential equation to a system comprising a PDE and an ordinary differential equation; see \cite{chaves-silva-etal_2014,chaves-silva-etal_2017}. 
Whether this method could be pursued successfully in order to study the LQ problem and devise a Riccati-based theory, as well, is a question that is left open here.

Finally, for context and pertinent background about the optimal control of {\em stochastic} Volterra 
equations, see for instance the recent \cite{bonaccorsi-confortola_2020} along with its references. 

%

%


%


\subsection{The integro-differential model, associated minimization problem} \label{ss:framework}
Let $H$ and $U$ be two separable Hilbert spaces, the {\em state} and {\em control} space, respectively.
Given $T>0$, we consider a linear Volterra integro-differential equation in the space $H$, and the corresponding Cauchy problem
\begin{equation}\label{e:state-eq_0}
\begin{cases}
w'(t)=Aw(t)+\displaystyle\int_0^t K(t-s) w(s)\,ds+Bu(t)\,, & t\in (0,T)
\\[1mm]
w(0)=w_0\in H\,, & 
\end{cases}
\end{equation}
under the following basic Assumptions on the operators $A$, $B$ and $K$ which appear in the state equation.
These operators describe the uncontrolled (or free) dynamics ($A$ and $K$) and the action of control functions ($B$), respectively; the ``memory kernel'' $K$ enters the convolution term which specifically accounts for a past history of the state variable.

 
\begin{assumptions}[\bf Basic Assumptions] \label{a:ipo_0} 
Let $H$, $U$ be separable complex Hilbert spaces. 
\begin{itemize}
\item
The closed linear operator $A\colon \cD(A)\subset H \to H$ is the infinitesimal generator of a strongly continuous semigroup $\{e^{tA}\}_{t\ge 0}$ on $H$, with $\|e^{tA}\|_{\cL(H)}\le C e^{\omega t}$ for all $t\ge 0$ and suitable constants $C, \omega$;

\item
$K\in L^2(0,T;\mathbb{R})$; 

\item 
$B\in \cL(U,H)$.
\end{itemize}

\end{assumptions} 


\begin{remarks}
\begin{rm}
Our analysis here focuses on a simple pattern for the integro-differential model.
In particular, (i) the hypothesis on the control operator $B$ covers the case of partial
differential equations (PDE) systems subject to {\em distributed} control.
The study of an integro-differential model in the presence of an {\em unbounded} control operator $B$ -- which is naturally brought about by boundary or point control actions -- is left to subsequent work.
(ii) On the other hand, the assumption that the memory kernel is real valued can be relaxed to $K\in L^2(0,T;\cL(H))$, with the computations carried out still valid, provided the operator $K(\cdot)$ commutes with the semigroup $e^{\cdot A}$.
(iii) The absence of (the realization of) a differential operator within the convolution term in \eqref{e:state-eq_0} is a restriction which is rendered milder in view of MacCamy's trick; see e.g. \cite[Sections~3.3 and 4.2]{pandolfi-book}. 
\end{rm}
\end{remarks}

\smallskip
To the state equation \eqref{e:state-eq_0} we associate the following quadratic functional over the preassigned time interval $[0,T]$:
\begin{equation} \label{e:cost}
J(u)=\int_0^T \left(\langle Qw(t),w(t)\rangle_H + \|u(t)\|_U^2\right)dt\,, 
\end{equation}
where the operator $Q$ simply satisfies 
\begin{equation*}
Q\in \cL(H)\,, \quad Q=Q^*\ge 0\,.
\end{equation*}
In Section~\ref{s:prelimin} we will derive a representation formula for the solution to the state equation in \eqref{e:state-eq_0}, as well as to the solutions to the family of Cauchy problems obtained taking as initial time $\tau\in (0,T)$ (in place of $\tau=0$), that is 
\eqref{e:family-cauchy} below.

\smallskip
The optimal control problem is formulated in the usual (classical) way.


\begin{problem}[\bf The optimal control problem] \label{p:problem-0}
Given $w_0\in H$, seek a control function $u\in L^2(0,T;U)$ which minimizes the
cost functional \eqref{e:cost}, where $w(\cdot)$ is the solution to \eqref{e:state-eq_0} 
corresponding to the control function $u(\cdot)$ (and with initial state $w_0$). 
\end{problem}

\begin{remark}
\begin{rm}
The solutions to \eqref{e:state-eq_0} are meant in a {\em mild} sense; see  
\eqref{d:mild-sln}.
\end{rm}
\end{remark}

\medskip

If $\tau\in (0,T)$ is given, and having set 
\begin{equation} \label{e:X_0}
\xi(\cdot) = w(\cdot)\big|_{[0,\tau]}\,, \quad \xi_0=w(\tau^+)\,, 
\quad X_0= \begin{pmatrix}\xi_0\\\xi(\cdot)\end{pmatrix}\,,
\end{equation}
we introduce the family of Cauchy problems
\begin{equation}\label{e:family-cauchy}
\begin{cases}
w'(t)=Aw(t)+\displaystyle{\int_\tau^t K(t-s) w(s)\,ds+\int_0^\tau K(t-s) \xi(s)\,ds}+Bu(t)\,, 
\; t\in (\tau,T)
\\[1mm]
w(\tau^+)=\xi_0
\end{cases}
\end{equation}
 and the associated cost functional
\begin{equation} \label{e:family-cost}
J_\tau(u,X_0)=\int_\tau^T \big(\langle Qw(t),w(t)\rangle_H + \|u(t)\|_U^2\big)\,dt\,.
\end{equation}

We will set $Y_\tau:= H\times L^2(0,\tau;H)$. 


\subsection{Main results}
Our main results establish the synthesis of the optimal control for our problem \eqref{e:state-eq_0}-\eqref{e:cost}, as certain operators $P_i$, $i=1,2,3$, which occur in its feedback representation -- as well as in the quadratic form that actualizes the optimal cost -- are entries of an operator matrix $P$ which is shown to be the unique solution to a quadratic Riccati-type equation.

 
\begin{theorem} \label{t:main}
With reference to the optimal control problem \eqref{e:family-cauchy}-\eqref{e:family-cost}, under
the Assumptions~\ref{a:ipo_0}, the following statements are valid for each $\tau\in (0,T)$.

\begin{enumerate}

\item[\bf S1.] 
For each $X_0\in Y_\tau$ there exists a unique optimal pair 
$(\hat{u}(\cdot,\tau;X_0),\hat{w}(\cdot,\tau;X_0))$
which satisfies 
\begin{equation*}
\hat{u}(\cdot,\tau;X_0)\in C([\tau,T],U)\,, \quad \hat{w}(\cdot,\tau;X_0)\in C([\tau,T],H)\,.
\end{equation*}

\smallskip

\item[\bf S2.] 
Given $t\in [\tau,T]$, the linear bounded operator  
$\Phi(t,\tau)\colon Y_\tau \longrightarrow Y_t$ defined by 
\begin{equation} \label{e:evolution-map} 
\Phi(t,\tau)X_0 :=
\begin{pmatrix}
\hat{w}(t,\tau;X_0)
\\
\hat{y}(\cdot)
\end{pmatrix}
\; \text{\small where} \; \hat{y}(\cdot)=
\begin{cases} 
\xi(\cdot) & \text{in $[0,\tau]$}
\\
\hat{w}(\cdot,\tau,X_0) & \text{in $[\tau,t]$}
\end{cases}
\end{equation}
is an evolution operator, namely, it satisfies 
\begin{equation*}
\Phi(t,t)=I\,, \qquad 
\Phi(t,\tau)=\Phi(t,\tau_1)\Phi(\tau_1,\tau) \quad \textrm{for \ $\tau\le \tau_1\le t\le T$.}
\end{equation*}

\smallskip

\item[\bf S3.]
There exist three bounded operators, denoted by $P_0(\tau)$, $P_1(\tau,s)$, $P_2(\tau,s,q)$ -- 
defined in terms of the optimal state and of the data of the problem
(see the expressions \eqref{e:riccati-ops} and \eqref{e:riccati-ops_2})  
--, such that the optimal cost is given by  
\begin{equation} \label{e:optimal-cost_1}
\begin{split}
\qquad J_\tau(\hat{u},X_0) &=\big\langle P_0(\tau) w_0,w_0\big\rangle_H
+ 2\text{Re}\, \int_0^\tau \langle P_1(\tau,s)\xi(s),w_0\big\rangle_H\,ds
\\[1mm]
& + \int_0^\tau \!\!\int_0^\tau \langle P_2(\tau,s,q)\xi(s),\xi(q)\big\rangle_H\,ds\,dq
\equiv \langle P(\tau)X_0,X_0\rangle_{Y_\tau}\,.
\end{split}
\end{equation}
$P_0(\tau)$ and  $P_2(\tau,s,q)$ are non-negative self-adjoint operators in the respective functional
spaces $H$ and $L^2(0,\tau;H)$.

\smallskip

\item[\bf S4.] 
The optimal control admits the following feedback representation
\begin{equation} \label{e:feedback}
\hat{u}(t,\tau;X_0)=-B^*P_0(t)\hat{w}(t;\tau,X_0)-\int_0^t B^*P_1(t,s)\hat{y}(s)\,ds\,, 
\; \tau\le t\le T\,,
\end{equation}
with $\hat{y}(\cdot)$ given by \eqref{e:evolution-map}.



\smallskip

\item[\bf S5.]  
The operators $P_0(t)$, $P_1(t,s)$, $P_2(t,s,q)$ -- as above in {\bf S3.} --
satisfy the following coupled system of equations, for every $t\in[0,T)$ and for any $x,y\in \cD(A)$:
\begin{equation} \label{e:DRE}
\begin{cases}
&\frac{d}{dt}\langle P_0(t)x,y\rangle_H +\langle P_0(t)x,Ay\rangle_H 
+ \langle Ax,P_0(t)y\rangle_H + \langle Qx,y\rangle_H  
\\[1mm] 
&\qquad\quad - \langle B^*P_0(t)x,B^*P_0(t)y\rangle_U 
+ \langle P_1(t,t)x,y\rangle_H +\langle x,P_1(t,t)y\rangle_H = 0 
\\[2mm]
&\frac{\partial}{\partial t} \langle P_1(t,s)x,y\rangle_H  + \langle P_1(t,s)x,Ay\rangle_H +
\langle K(t-s)x,P_0(t)y\rangle_H 
\\[1mm] 
&\qquad\quad
+ \langle P_2(t,s,t)x,y\rangle_H - \langle B^*P_1(t,s)x,B^*P_0(t)y\rangle_U =0
\\[2mm]
& \frac{\partial}{\partial t}\langle P_2(t,s,q)x,y\rangle_H
+\langle P_1(t,s)x,K(t-q)y\rangle_H
+\langle K(t-s)x,P_1(t,q)y\rangle_H
\\[1mm]
&\qquad\quad  -\langle B^*P_1(t,s)x,B^*P_1(t,q)y\rangle_U=0
\end{cases}
\end{equation}
with final conditions
\begin{equation*}
P_0(T)=0\,, \; P_1(T,s)=0\,, P_2(T,s,q)=0\,.
\end{equation*} 

\end{enumerate}

\end{theorem}

\bigskip
The coupled system of (four) equations satisfied by the operators $P_0(t)$, $P_1(t,\cdot)$, $P_1(t,:)^*$ and $P_2(t,\cdot,:)$ -- with the one for $P_1(t,:)^*$ derived later in the paper, see \eqref{e:P_1*} -- can be shown to be equivalent to a single equation satisfied by a matrix operator
\begin{equation} \label{e:big-op}
P(t):=
\begin{pmatrix}
P_0(t) & P_1(t,\cdot)
\\
P_1(t,:)^* & P_2(t,\cdot,:)
\end{pmatrix}
\end{equation}
in the product space $H\times L^2(0,t;H)$. 
Such a unified form of the equation, that is \eqref{e:big-riccati} below,
\begin{itemize}
\item
renders more explicit its Riccati-type nature,
\item
reduces to a standard Riccati equation in the absence of memory.
\end{itemize}


\begin{theorem}[Well-posedness for the Riccati equation] \label{t:big-riccati}
With reference to the optimal control problem \eqref{e:family-cauchy}-\eqref{e:family-cost}, under the Assumptions~\ref{a:ipo_0}, let $P_0(t)$, $P_1(t,\cdot)$ and $P_2(t,\cdot,:)$ as from the statement S3. of Theorem~\ref{t:main}.
Then, the (matrix) operator $P(t)$ defined by \eqref{e:big-op} is the unique solution of the following quadratic equation with unbounded coefficients:
\begin{equation} \label{e:big-riccati}
\frac{d}{dt} P(t)+P(t)(\cA+\cK_1(t)+\cD_{1,t})+(\cA^*+\cK_2(t)+\cD_{2,t})P(t)
- P(t)\cB\cB^*P(t)+\cQ=0\,,
\end{equation}
where we set
\begin{equation*}
\begin{split}
&\cA:=
\begin{pmatrix} 
A & 0
\\
0 & 0
\end{pmatrix},
\qquad
\cB:=
\begin{pmatrix} 
B & 0
\\
0 & 0
\end{pmatrix},
\qquad
\cQ:=
\begin{pmatrix} 
Q & 0
\\
0 & 0
\end{pmatrix},
\\[1mm]
& \cK_1(t):=
\begin{pmatrix} 
0 & K(t-\cdot)
\\
0 & 0
\end{pmatrix},
\qquad
\cK_2(t):=
\begin{pmatrix} 
0 & 0
\\
K(t-:) & 0
\end{pmatrix}, 
\\[1mm]
& \cD_{1,t}:=
\begin{pmatrix} 
0 & 0
\\
\delta_t(:) & 0
\end{pmatrix}, 
\qquad 
\cD_{2,t}:=
\begin{pmatrix} 
0 & \delta_t(\cdot)
\\
0 & 0
\end{pmatrix},
\end{split}
\end{equation*}
and $\delta_t(\cdot)$ denotes the Dirac delta distribution: to wit,
$\delta_t f(\cdot)=f(t)$. 
  
\end{theorem}



\subsection{Notation}
The concise notation $\partial_\tau$ ($\partial_t$, etc.) in place of 
$\frac{\partial}{\partial\tau}$ ($\frac{\partial}{\partial t}$, etc., respectively), 
will be adopted throughout.


\subsection{An overview of the paper}
To a large extent, the proof of Theorem~\ref{t:main} retraces the principal steps of the proofs carried out in the study of the LQ problem for memoryless control systems; see \cite{las-trig-redbooks}.
However, the solution formula corresponding to the integro-differential initial value problem naturally accounts for the more involved computations {\em at any step} of the line of investigation.
Also, a difference -- and technically challenging element -- in comparison with the memoryless case stands in the fact that the optimal cost operator is not readily identified in a first formula which relates the optimal control to the optimal state, as it follows from the optimality condition.

Furthermore, recasting the system of quadratic equations satisfied by the optimal cost operators
as a single Riccati-type equation in the state space as well as proving uniqueness demand that nontrivial analytical matters are addressed and overcome.

In the following section, i.e. in Section \ref{s:prelimin}, we derive an explicit representation for the solutions to the integro-differential problem in terms
of the initial state $X_0$ (which actually comprises the state at an initial time
and the memory up to it) and the control function $u(\cdot)$.
The said representation, achieved by solving a certain Volterra equation of the second kind, generalizes the well-known {\em input-to-state formula} for (memoryless) linear control systems.

In Section~\ref{s:transition} we prove the first two statements (namely, S1. and S2.) of Theorem~\ref{t:main}.
We not only infer existence of the unique optimal pair, but also pinpoint the transition properties fulfilled by both the optimal state and the optimal control.

In Section~\ref{s:feedback} we deal with the statements S3. and S4. of Theorem~\ref{t:main}.
We find readily that the representation of the optimal cost as a quadratic form involves three operators $P_i$, $i=0,1,2$, rather than just one.
Attaining certain alternative expressions for the said operators -- as it is pursued in 
Lemma~\ref{l:key-lemma} -- is a critical (and nontrivial) step in our analysis, since it
enables us to establish a sought {\em closed-loop} form of the optimal control;
see Proposition~\ref{p:closedloop}.

Section~\ref{s:riccati-eqns} is entirely devoted to the proof of statement S5.,
namely, to derive the coupled system of three differential equations satisfied -- uniquely -- by
the operators $P_i$, $i=0,1,2$. 

Section~\ref{s:uniqueness} focuses on the proof of Theorem~\ref{t:big-riccati}.
Its existence and uniqueness parts rely upon two instrumental results (Proposition~\ref{p:DRE_equiv} and Lemma~\ref{l:lsc}) which resolve respective technical points.

In Appendix~\ref{a:appendix} we gather several analytical results which are primarily utilized in the proofs of Lemma~\ref{l:key-lemma} and of the fundamental assertion S5.~of Theorem \ref{t:main},
as well as for the question of uniqueness within Theorem~\ref{t:big-riccati}.


\section{Preliminaries. A representation formula for the solutions} \label{s:prelimin}
The starting point for a study of the optimal control problem with quadratic functionals
for (memoryless) linear differential systems of general form $y'=Ay+Bu$ is a representation formula for the {\em mild} solutions corresponding to an initial state $y_0:=y(\tau)$ and a control function $u(\cdot)$, that is
\begin{equation*}
y(t)= e^{(t-\tau)A} y_0 + \int_\tau^t e^{(t-s)A} B u(s)\,ds\,, \quad t\in [\tau,T)\,; 
\end{equation*}
see \cite{las-trig-redbooks}.
Then, as it is well known, the analyses of boundary control systems governed by PDE
split apart, depending on the distinct regularity properties of the (so called) {\em input-to-state map}
\begin{equation*}
L\colon L^2(\tau,T;U) \ni u(\cdot) \longrightarrow (Lu(\cdot))(t):=\int_\tau^t e^{(t-s)A} B u(s)\,ds 
\end{equation*}
that occurs in the said formula, as well as of its adjoint, in accordance with a parabolic
or hyperbolic character of the free dynamics.

In the present context, the following definition appears natural.


\begin{definition} \label{d:mild-sln}
We say that a function $w(t)$ is a mild solution to the Cauchy problem \eqref{e:family-cauchy} corresponding to an initial datum $X_0$ -- that subsumes the state $\xi_0$ at time $\tau\in (0,T)$
and the past history $\xi(\cdot)$ on $[0,\tau)$ -- and a control function $u(\cdot)\in L^2(\tau,T;U)$, if it belongs to $L^2(\tau,T;H)$ and satisfies the integral equation 
\begin{equation}\label{e:integral-eq_0}
\begin{split}
& w(t,\tau,X_0)\equiv w(t)=e^{(t-\tau)A}\xi_0 
+ \int_\tau^t e^{(t-s)A} \int_\tau^s K(s-\sigma) w(\sigma)\,d\sigma\,ds
\\[1mm]
& \qquad 
+ \int_\tau^t e^{(t-s)A} \int_0^\tau K(s-\sigma) \xi(\sigma)\,d\sigma\,ds
+ \int_\tau^t e^{(t-s)A}Bu(s)\,ds
\end{split}
\end{equation}
a.e. on $[\tau,T]$.
\end{definition}

\begin{remark}
\begin{rm}
We note here that if a mild solution $w$ exists according to the above definition (and
under the Assumptions \ref{a:ipo_0}), then $w\in C([\tau,T],H)$. 
\end{rm}
\end{remark}

\medskip
The theory of linear Volterra equations allows to achieve a representation formula for the mild solutions to the family of Cauchy problems \eqref{e:family-cauchy} (depending on the parameter $\tau$).
The proof is not difficult, yet it is given for completeness and the readers' convenience.


\begin{proposition} \label{p:representation}
For any $X_0$ as in \eqref{e:X_0} and $u\in L^2(\tau,T;U)$, the (controlled) integro-differential problem \eqref{e:family-cauchy} admits a unique mild solution $w=w(t)$,
given by 
\begin{equation} \label{e:represent}
w(t)=w(t;\tau,X_0) 
= F(t,\tau)\xi_0 + \int_0^\tau M(t,\sigma,\tau) \xi(\sigma)\,d\sigma
+ \int_\tau^t F(t,s)Bu(s)\,ds, \ t\in [\tau,T]\,,
\end{equation}
where
\begin{subequations} \label{e:vari-ops}
\begin{align}
F(t,\tau) &:= e^{(t-\tau)A}-\int_\tau^t R(t-s) e^{(s-\tau)A}\,ds\,, & \tau\le t,
\label{e:F}
\\
M(t,\sigma,\tau)&:= G(t,\sigma,\tau)- \int_\tau^t R(t-s) G(s,\sigma,\tau)\, ds\,, & \sigma\le \tau\le t\,,
\label{e:M}
\\
G(t,\sigma,\tau)&:= \mu(t-\sigma)-e^{(t-\tau)A}\mu(\tau-\sigma)\,, & \sigma\le \tau\le t,
\label{e:G}
\\
\mu(t)& := \int_0^t e^{(t-s)A} K(s)\,ds\,,
\label{e:mu}
\end{align}
\end{subequations}
while $R(t)$ is the unique solution to the Volterra equation (of the second kind)
\begin{equation} \label{e:resolvent-op}
R(t) - \int_0^t \mu(t-s) R(s) \,ds = -\mu(t)\,, \quad t>0,
\end{equation}
explicitly given by
\begin{equation} \label{e:iterated}
R(t) = -\mu(t) - \int_0^t \mu(t-s)\mu(s)\,ds - \int_0^t\mu(t-s)\int_0^\sigma\mu(\sigma-s)\mu(s)\,ds\,d\sigma - \dots\,.
\end{equation}

\end{proposition}


\begin{proof}
With \eqref{e:integral-eq_0} as a starting point, let $w=w(t)$ an $L^2$-in time solution.
We rewrite the second summand in its right hand side by exchanging the order of integration:
\begin{equation}\label{e:third}
\begin{split} 
&\int_\tau^t e^{(t-s)A} \int_\tau^s K(s-\sigma) w(\sigma)\,d\sigma\,ds
= \int_\tau^t \Big[\int_\sigma^t e^{(t-s)A} K(s-\sigma)\,ds\Big]\, w(\sigma)\,d\sigma
\\
& \qquad =\int_\tau^t \Big[\int_0^{t-\sigma} e^{(t-\sigma-\lambda)A} K(\lambda) \,d\lambda\Big]\,
w(\sigma) \,d\sigma 
=  \int_\tau^t \mu(t-\sigma) w(\sigma)\,d\sigma\,,
\end{split}
\end{equation}
having set $\mu(\cdot)$ as in \eqref{e:mu}.
Analogously, the third summand in the right hand side of \eqref{e:integral-eq_0} becomes
\begin{equation}\label{e:fourth}
\begin{split} 
&\int_\tau^t e^{(t-s)A} \int_0^\tau K(s-\sigma) \xi(\sigma)\,d\sigma\,ds
= \int_0^\tau \Big[\int_\tau^t e^{(t-s)A} K(s-\sigma) \xi(\sigma)\,ds\Big]\,d\sigma
\\
& \quad =  \int_0^\tau \Big[\int_{\tau-\sigma}^{t-\sigma} e^{(t-\sigma-\lambda)A} K(\lambda)\,d\lambda\Big]\, \xi(\sigma)d\sigma
\\ 
& \quad 
=  \int_0^\tau \Big[\mu(t-\sigma)-e^{(t-\tau)A}\mu(\tau-\sigma)\Big] \xi(\sigma)\,d\sigma\,.
\end{split}
\end{equation}
Returning to \eqref{e:integral-eq_0}, in the light of \eqref{e:third} and \eqref{e:fourth}, we find
\begin{equation}\label{e:integral-eq_1}
\begin{split}
& w(t,\tau,X_0)\equiv w(t)=e^{(t-\tau)A}\xi_0 
+ \int_\tau^t \mu(t-\sigma) w(\sigma)\,d\sigma
\\[1mm]
& \qquad + \int_0^\tau \Big[\mu(t-\sigma)-e^{(t-\tau)A}\mu(\tau-\sigma)\Big] \xi(\sigma)\,d\sigma
+ \int_\tau^t e^{(t-s)A}Bu(s)\,ds\,,
\end{split}
\end{equation}
which is a simple Volterra equation of the form 
\begin{equation}\label{e:volterra}
w(t)-\int_\tau^t \mu(t-\sigma) w(\sigma)\,d\sigma=\cF(t)\,,
\end{equation}
with $\cF(t)$ depending on the initial datum, the past history and the control, 
specifically given by
\begin{equation*}
\cF(t):=e^{(t-\tau)A}\xi_0 
+ \int_0^\tau \Big[\mu(t-\sigma)-e^{(t-\tau)A}\mu(\tau-\sigma)\Big] \xi(\sigma)\,d\sigma
+ \int_\tau^t e^{(t-s)A}Bu(s)\,ds\,.
\end{equation*}
Noticing the presence of the operator $G(t,\sigma,\tau)$ defined by \eqref{e:G},
$\cF(t)$ is rewritten more neatly as follows:
\begin{equation}\label{e:calF}
\cF(t):=e^{(t-\tau)A}\xi_0  + \int_0^\tau G(t,\sigma,r) \xi(\sigma)\,d\sigma
+ \int_\tau^t e^{(t-s)A}Bu(s)\,ds\,.
\end{equation}
Thus, it is well known that the solution to the Volterra equation \eqref{e:volterra} is
given by 
\begin{equation}\label{e:soln-volterra}
w(t) = \cF(t) -\int_\tau^t R(t-s) \cF(s)\, ds\,,
\end{equation}
where $R(t)$ is the resolvent kernel of \eqref{e:volterra}, that is the unique solution to the integral equation \eqref{e:resolvent-op}; see e.g. \cite[Chapter~5]{corduneanu}.
With $\cF(t)$ given by \eqref{e:calF}, then \eqref{e:soln-volterra} 
reads as
\begin{equation*}
\begin{split}
w(t) &= e^{(t-\tau)A}\xi_0 
+ \int_0^\tau G(t,\sigma,\tau) \xi(\sigma)\,d\sigma
+ \int_\tau^t e^{(t-s)A}Bu(s)\,ds
\\[1mm]
& \qquad
- \int_\tau^t R(t-s) e^{(s-\tau)A}\xi_0\,ds 
- \int_\tau^t R(t-s) \int_0^\tau G(s,\sigma,\tau) \xi(\sigma)\,d\sigma\,ds
\\[1mm]
& \qquad
- \int_\tau^t R(t-s) \int_\tau^s e^{(s-\sigma)A}Bu(\sigma)\,d\sigma\,ds\,.
\end{split}
\end{equation*}
Thus, if we set $F(t,\tau)$ and $M(t,\sigma,\tau)$ as in \eqref{e:F} and \eqref{e:M},
respectively, we finally attain -- for the solution to the integral equation \eqref{e:integral-eq_1},
which is equivalent to the original one \eqref{e:integral-eq_0} -- 
the representation formula \eqref{e:represent}.

A simple verification confirms that the function in \eqref{e:represent} is indeed the unique mild solution of \eqref{e:integral-eq_0}, which concludes the proof.
\end{proof}


\section{The unique optimal pair. Transition properties, regularity, statements S1.~and
S2.~of Theorem~\ref{t:main}}
\label{s:transition}
We begin this section by proving that, given $\tau\in (0,T)$, every solution $w(t;\tau,X_0)$ to the integro-differential problem \eqref{e:family-cauchy} satisfies a transition property, just like in the memoryless case.

Next, we show that the cost functional \eqref{e:family-cost} is a quadratic form in the space $Y_\tau=H\times L^2(0,\tau;H)$.
This brings about the existence of a unique optimal control $\hat{u}(t,\tau,X_0)$, along with a first (pointwise in time) representation of 
$\hat{u}(\cdot)$ in dependence on the optimal state $\hat{w}(\cdot)$.
The latter is a straightforward outcome of the optimality condition.

A further analysis allows then to prove that the optimal control inherits a transition property from the optimal state, as well.


\subsection{Transition property for the state variable}
Let $w(t):=w(t;\tau,X_0)$, $t\in [\tau,T]$, be the mild solution to the integro-differential equation \eqref{e:family-cauchy} corresponding to an initial datum $X_0$ and a control function $u(\cdot)$.
For $\tau_1\in (\tau,t)$, define
\begin{equation} \label{e:X_1}
X_1=\begin{pmatrix}y_0\\[1mm] y(\cdot)\end{pmatrix}\,, 
\quad \text{\small with} \quad
y_0=w(\tau_1^+)\,, \quad 
y(\cdot) =\begin{cases}\xi(\cdot) & \text{in $[0,\tau]$}
\\
w(\cdot,\tau,X_0) & \text{in $(\tau,\tau_1]$.}
\end{cases}
\end{equation}

Then, the following result holds true.


\begin{proposition}
The following transition property
\begin{equation*}
w(t;\tau,X_0)=w(t;\tau_1,X_1) \qquad \forall t\in (\tau_1,T)
\end{equation*}
is valid.
\end{proposition}

\begin{proof}
The function $w_1(t):=w(t;\tau_1,X_1)$ solves the initial value problem 
\begin{equation} \label{e:eq-for-w_1}
\begin{cases}
w_1'(t)=Aw_1(t)+\displaystyle\int_{\tau_1}^t K(t-s) w_1(s)\,ds
+\displaystyle\int_0^{\tau_1} K(t-s) y(s)\,ds+Bu(t)
\\[3mm]
\qquad \;\; = Aw_1(t)+ Bu(t) +\displaystyle\int_{\tau_1}^t K(t-s) w_1(s)\,ds
+\displaystyle\int_0^{\tau} K(t-s) \xi(s)\,ds
\\[3mm]
\myspace +\displaystyle\int_{\tau}^{\tau_1} K(t-s) w(s)\,ds  
\\[3mm]
w_1(\tau_1^+)=y_0=w(\tau_1^+) 
\end{cases}
\end{equation}
(in fact its mild form).
Thus the function $z(t)=w_1(t)-w(t)$ is a mild solution to 
\begin{equation*}
\begin{cases}
z'(t)=Az(t)+\displaystyle \int_{\tau_1}^t K(t-s) w_1(s)\,ds -\int_{\tau}^t K(t-s) w(s)\,ds +\int_{\tau}^{\tau_1} K(t-s) w(s)\,ds
\\[3mm]
\qquad =Az(t)+\displaystyle\int_{\tau_1}^t K(t-s) z(s)\,ds  
\\[3mm]
z(\tau_1^+)=0\,;
\end{cases}
\end{equation*}
then   
\begin{equation*}
z(t) = \int_{\tau_1}^t e^{(t-\sigma)A} \int_{\tau_1}^\sigma K(\sigma-s) z(s)\, ds\,d\sigma
= \int_{\tau_1}^t \Big[\int_{s}^t e^{(t-\sigma)A}  K(\sigma-s) \,d\sigma\Big]\,z(s)\,ds\,.
\end{equation*}

Therefore,
\begin{equation*}
\|z(t)\|_H\le C\,e^{\omega T}\|K(\cdot)\|_{L^1(0,T;\mathbb{R})} \int_{\tau_1}^t \|z(s)\|_H\,ds\,, \quad t\in [\tau_1,T]\,,
\end{equation*}
which implies $\|z(t)\|_H\equiv 0$ on $[\tau_1,T]$ by the Gronwall Lemma.
The conclusion $w_1\equiv w$ on $[\tau_1,T]$ follows.
 
\end{proof}


\subsection{The optimal pair. Proof of the statement S1.} 
Recall the representation formula \eqref{e:represent} for the solution to the Cauchy problem \eqref{e:family-cauchy}.
Just like in the study of memoryless control systems, it is useful to introduce the operator 
$L_\tau\colon L^2(\tau,T;U) \longrightarrow L^2(\tau,T;H)$ defined as follows,
\begin{equation} \label{e:L_tau}
[L_\tau u(\cdot)](t):= \int_\tau^t F(t,\sigma) B u(\sigma)\,d\sigma\,, \qquad t\in [\tau,T]\,,
\end{equation}
along with its adjoint $L_\tau^*\colon L^2(\tau,T;H)\longrightarrow L^2(\tau,T;U)$
that is deduced readily:
\begin{equation}\label{e:L_tau-star}
[L_\tau^*g(\cdot)](\sigma):= \int_\sigma^T B^*F(t,\sigma)^* g(t)\,dt\,, \qquad \sigma\in [\tau,T]\,.
\end{equation}
We will use the abbreviated notations $[L_\tau u](t)$ and even the neat $L_\tau u(t)$, in place of
$[L_\tau u(\cdot)](t)$, etc.
Then, \eqref{e:represent} reads as
\begin{equation} \label{e:formula_1}
w(t,\tau,X_0)\equiv w(t) = F(t,\tau)\xi_0 + \int_0^\tau M(t,\sigma,\tau) \xi(\sigma)\,d\sigma
+ L_\tau u(t)\,.
\end{equation}
By inserting the expression \eqref{e:formula_1} of $w(t)$ in the cost functional \eqref{e:family-cost},
we obtain readily
\begin{equation} \label{e:q-form}
J_\tau(u,X_0)= \langle \cM_\tau X_0,X_0\rangle_{Y_\tau}
+ 2 \text{Re}\, \langle N_\tau X_0,u\rangle_{L^2(\tau,T;U)}
+ \langle \Lambda_\tau u,u\rangle_{L^2(\tau,T;U)}\,,
\end{equation}
where
\begin{equation} \label{e:maiuscoli}
\begin{split}
\langle \cM_\tau X_0,X_0\rangle_{Y_\tau}&:=\int_\tau^T \langle Q E(t,\tau)X_0,E(t,\tau)X_0\rangle_H\,dt
=\int_\tau^T \langle Q F(t,\tau)\xi_0,F(t,\tau)\xi_0\rangle_H\,dt
\\[1mm]
& \qquad + 2\text{Re}\,\int_\tau^T \left\langle Q F(t,\tau)\xi_0,\int_0^\tau M(t,\sigma,\tau) \xi(\sigma)\,d\sigma\right\rangle_H\,dt
\\[1mm]
& \qquad +\int_\tau^T \left\langle \int_0^\tau M(t,\sigma,\tau)\xi(\sigma)\,d\sigma,
\int_0^\tau M(t,q,\tau) \xi(q)\,dq \right\rangle_H\,dt
\\[3mm]
N_\tau X_0 &:=\big[L_\tau^* Q E(\cdot,\tau)X_0\big](\cdot)
\\[3mm]
\Lambda_\tau & :=I+L_\tau^*QL_\tau\,,
\end{split}
\end{equation}
having set 
\begin{equation} \label{e:E}
E(t,\tau)X_0:= F(t,\tau)\xi_0 + \int_0^\tau M(t,\sigma,\tau)\xi(\sigma)\,d\sigma
\end{equation}
for the sake of brevity (although this abbreviated notation will seldom occur)
and where $I$ denotes the identity operator on $L^2(\tau,T;U)$.

A pretty standard argument is invoked now: notice that from the assumption $Q\ge 0$ it follows $\Lambda_\tau\ge I$; namely, the cost functional is coercive in the space $\cU_\tau=L^2(\tau,T;U)$
of admissible controls, and hence there exists a unique optimal control minimizing the cost \eqref{e:family-cost}.
The optimality condition
\begin{equation*}
\Lambda_\tau \hat{u}+N_\tau X_0=0
\end{equation*}
yields, on the one side, 
\begin{equation} \label{e:from-optimality}
\hat{u}=-\Lambda_\tau^{-1} N_\tau X_0\,.
\end{equation}
On the other side, recalling \eqref{e:maiuscoli} and rewriting explicitly $\Lambda_\tau$, 
we see that
\begin{equation*}
\hat{u}+L_\tau^*QL_\tau\hat{u}+L_\tau^* Q E(\cdot,\tau)X_0 =0\,,
\end{equation*}
that is 
\begin{equation*}
\hat{u}=- L_\tau^*Q\big[E(\cdot,\tau) X_0 + L_\tau\hat{u}\big]=- L_\tau^*Q\hat{w}\,.
\end{equation*}
This is nothing but a first representation of the optimal control in terms of the optimal state: 
\begin{equation} \label{e:feedback_0}
\hat{u}(t,\tau;X_0)=- [L_\tau^*Q\hat{w}(\cdot,\tau;X_0)](t)
= -\int_t^T B^*F(\sigma,t)^*Q\hat{w}(\sigma,\tau;X_0)\,d\sigma\,.
\end{equation}
We note that, as $w(\cdot)$ is a continuous function in view of \eqref{e:represent},
\eqref{e:feedback_0} establishes that the optimal control is continuous in time as well.
Thus, the statement S1. is proved.


\subsection{Transition property for the optimal pair. Proof of the statement S2.}
In order to infer that the transition property fulfilled by the optimal state is inherited by the optimal control, we follow an argument which is pretty standard in the case of memoryless control systems.
Given $\tau_1>\tau$, and with $X_1$ as in \eqref{e:X_1} ($y_0$ and $y(\cdot)$ are defined therein),
we associate to the state equation with initial time $\tau_1$ and initial state $X_1$ the cost functional $J_{\tau_1}(u,X_1)$. 
With $\hat{w}(\cdot,\tau;X_0)$ the optimal state of the original control problem, assuming that
$\hat{w}(\cdot;\tau,X_0)$ restricted to $[\tau_1,T]$ is optimal for $J_{\tau_1}$
as well, then it follows from \eqref{e:feedback_0} that 
\begin{equation*}
\begin{split}
\hat{u}(t;\tau_1X_1)&= -\int_t^T B^*F(\sigma,t)^*Q\hat{w}(\sigma;\tau_1,X_1)\,d\sigma
\\[1mm]
&= -\int_t^T B^*F(\sigma,t)^*Q\hat{w}(\sigma;\tau,X_0)\,d\sigma
=\hat{u}(\cdot;\tau,X_0)\,, \qquad \tau_1< t\le T\,;
\end{split}
\end{equation*}
to wit, the optimal control satisfies a transition property as well.
In the following Lemma we prove that indeed the assumed condition holds true.


\begin{lemma}
Let $(\hat{u}(\cdot;\tau,X_0),\hat{w}(\cdot;\tau,X_0))$ be the optimal pair of
problem \eqref{e:family-cauchy}-\eqref{e:family-cost}. 
Then $(\hat{u}|_{[\tau_1,T]},\hat{w}|_{[\tau_1,T]})$ is the optimal pair of the minimization problem 
with functional $J_{\tau_1}(u,X_1)$.
\end{lemma}

\begin{proof}
Consider the optimal control problem \eqref{e:family-cauchy}-\eqref{e:family-cost}. 
By definition,  
\begin{equation*}
J_{\tau}(\hat{u},X_0)\le J_{\tau}(u,X_0) \qquad \forall u\in L^2(\tau,T;U)\,. 
\end{equation*}
Given $u\in L^2(\tau_1,T;U)$, let $w$ be the state which corresponds to the control 
function $u$ and to the initial datum $X_1$ defined by \eqref{e:X_1}: then 
$w$ satisfies (the same Cauchy problem as \eqref{e:eq-for-w_1})
\begin{equation} 
\begin{cases}
w'(t)-\displaystyle\int_{\tau_1}^t K(t-s) w(s)\,ds
=Aw(t)+ Bu(t) +\displaystyle\int_0^{\tau_1} K(t-s) y(s)\,ds
\\[3mm]
w_1(\tau_1)=\hat{w}(\tau_1)\,. 
\end{cases}
\end{equation}
Let us introduce
\begin{equation*}
\overline{u}(\cdot)=\begin{cases}
\hat{u}(\cdot) & \text{in $[\tau,\tau_1]$}
\\[1mm] u(\cdot) & \text{in $[\tau_1,T]$} 
\end{cases}\,,
\qquad 
\overline{w}(\cdot)=\begin{cases}
\hat{w}(\cdot) & \text{in $[\tau,\tau_1]$}
\\[1mm] w(\cdot) & \text{in $[\tau_1,T]$} 
\end{cases}\,.
\end{equation*}

If $t\in [\tau,\tau_1]$, then $\overline{w}(\cdot)$ is such that
\begin{equation} \label{e:wsegnato_a}
\begin{split} 
\overline{w}'(t)-\int_{\tau}^t K(t-s) \overline{w}(s)\,ds-A\overline{w}(t)
&= \hat{w}'(t)-\int_{\tau}^t K(t-s) \hat{w}(s)\,ds-A\hat{w}(t)
\\[1mm]
& = B\hat{u}(t) +\int_0^{\tau} K(t-s) \xi(s)\,ds\,. 
\end{split}
\end{equation}
When $t\in [\tau_1,T]$, one has 
\begin{equation} \label{e:wsegnato_b}
\begin{split} 
&\overline{w}'(t)-\int_{\tau}^t K(t-s) \overline{w}(s)\,ds-A\overline{w}(t)
\\[1mm]
& \qquad
= w'(t)-\int_{\tau}^{\tau_1} K(t-s) \hat{w}(s)\,ds-\int_{\tau_1}^t K(t-s) w(s)\,ds -A w(t)
\\[1mm]
& \qquad
= -\int_{\tau}^{\tau_1} K(t-s) \hat{w}(s)\,ds + Bu(t) +\int_0^{\tau_1} K(t-s) y(s)\,ds
\\[1mm]
& \qquad
= B u(t) +\int_0^\tau K(t-s) \xi(s)\,ds\,, 
\end{split}
\end{equation}
instead.
In view of \eqref{e:wsegnato_b} and \eqref{e:wsegnato_a}, we find that $\overline{w}$
satisfies
\begin{equation}
\overline{w}'(t)-\int_{\tau}^t K(t-s) \overline{w}(s)\,ds-A\overline{w}(t)=
B\overline{u}(t) +\int_0^{\tau} K(t-s) \xi(s)\,ds \quad \forall t\in [\tau, T]\,,
\end{equation}
which means that $\overline{w}(\cdot)$ is the state corresponding to the control $\overline{u}(\cdot)$, with initial state $X_0$, in $[\tau,T]$. 
Therefore we have 
\begin{equation*}
J_\tau(\hat{u},X_0)\le J_{\tau}(\overline{u},X_0)\,, 
\end{equation*}
which reads as
\begin{equation*} 
\int_\tau^T \left[\langle Q\hat{w}(t),\hat{w}(t)\rangle_H + \|\hat{u}(t)\|_U^2\right]dt
\le  \int_\tau^T \left[\langle Q\overline{w}(t),\overline{w}(t)\rangle_H 
+ \|\overline{u}(t)\|_U^2\right]dt\,.
\end{equation*}
Deleting the integrals between $\tau$ and $\tau_1$ in both sides, as $\hat{w}(\cdot)$
and $\overline{w}(\cdot)$ coincide on $[\tau,\tau_1]$, we obtain
\begin{equation*} 
\int_{\tau_1}^T \left[\langle Q\hat{w}(t),\hat{w}(t)\rangle_H + \|\hat{u}(t)\|_U^2\right]dt
\le  \int_{\tau_1}^T \left[\langle Q w(t),w(t)\rangle_H 
+ \|u(t)\|_U^2\right]dt\,.
\end{equation*}
that is
\begin{equation*}
J_{\tau_1}(\hat{u}|_{[\tau_1,T]},X_1)\le J_{\tau_1}(u,X_1)\,. 
\end{equation*}
Since $u(\cdot)$ was an arbitrarily chosen admissible control in $[\tau_1,T]$, 
it follows that $$\big(\hat{u}|_{[\tau_1,T]},\hat{w}|_{[\tau_1,T]}\big)$$
is the optimal pair for the minimization problem in the time interval $[\tau_1,T]$, 
whose functional is $J_{\tau_1}(u,X_1)$. This concludes the proof.

\end{proof}

In light of the above Lemma, the statement S2. of Theorem~\ref{t:main} is established.


\begin{proposition} \label{p:transition}
The transition property
\begin{equation*}
\begin{cases}
\hat{u}(t;\tau,X_0)=\hat{u}(t;\tau_1,X_1)
\\[1mm]
\hat{w}(t;\tau,X_0)=\hat{w}(t;\tau_1,X_1)
\end{cases} \qquad \forall t\in (\tau_1,T)
\end{equation*}
holds true for the optimal control, just like it is valid for the optimal state.
\end{proposition}

We set now, for $t>\tau$,
\begin{equation} \label{e:transition-map}
\Phi(t,\tau)X_0:= \begin{pmatrix}
\hat{w}(t;\tau,X_0)
\\
\hat{y}(\cdot)
\end{pmatrix}
\end{equation}
with $\hat{y}(\cdot)$ as in \eqref{e:evolution-map}.
It is seen immediately that $\Phi(t,\tau)\colon Y_\tau\longrightarrow Y_t$
and that it is a linear map, owing to \eqref{e:from-optimality} combined with
\eqref{e:formula_1}.  
And notably, in view of Proposition~\ref{p:transition}, the transition property
\begin{equation*}
\Phi(t,\tau)=\Phi(t,\tau_1)\,\Phi(\tau_1,\tau) \qquad \forall \tau\in (\tau, t)
\end{equation*}
holds true.
Indeed, let $\tau<\tau_1<t$. Set 
\begin{equation*}
X_0=
\begin{pmatrix}
\xi_0
\\
\xi(\cdot)
\end{pmatrix}\,,
\quad 
X_1=\Phi(\tau_1,\tau)X_0= 
\begin{pmatrix}
\hat{w}(\tau_1,\tau;X_0)
\\[1mm] 
\hat{y}(\cdot)
\end{pmatrix}
\end{equation*}
where $\hat{y}(\cdot)$ is defined by
\begin{equation}\label{e:opt_X_1}
\hat{y}(\cdot) =\begin{cases}\xi(\cdot) & \text{in $[0,\tau]$}
\\
\hat{w}(\cdot,\tau,X_0) & \text{in $(\tau,\tau_1]$.}
\end{cases}
\end{equation}
Then, we have 
\begin{equation*}
\Phi(t,\tau_1)X_1= 
\begin{pmatrix}
\hat{w}(t,\tau_1;X_1)
\\[1mm] 
z(\cdot)
\end{pmatrix}\,,
\quad \text{\small where} \quad
z(\cdot)=
\begin{cases}\hat{y}(\cdot) & \text{in $[0,\tau_1]$}
\\
\hat{w}(\cdot,\tau_1;X_1) & \text{in $[\tau_1,t]$}
\end{cases}
\end{equation*}
that is
\begin{equation*}
z(\cdot)=
\begin{cases}
\xi & \text{in $[0,\tau]$} 
\\
\hat{w}(\cdot,\tau;X_0) & \text{in $[\tau,\tau_1]$}
\\
\hat{w}(\cdot,\tau_1;X_1) & \text{in $[\tau_1,t]$.}
\end{cases}
\end{equation*}
Thus, since in view of Proposition~\ref{p:transition} $\hat{w}(\cdot,\tau_1;X_1)=\hat{w}(\cdot,\tau;X_0)$
in $[\tau_1,T]$, we infer
\begin{equation*}
\Phi(t,\tau)X_0= 
\begin{pmatrix}
\hat{w}(t,\tau;X_0)
\\[1mm] 
X(\cdot)
\end{pmatrix}\,,
\quad \text{where} \quad 
X(\cdot)=
\begin{cases}
\xi & \text{in $[0,\tau]$} 
\\
\hat{w}(\cdot,\tau;X_0) & \text{in $[\tau,t]$,}
\end{cases}
\end{equation*}
i.e. $X(\cdot)\equiv z$; namely,
\begin{equation*}
\Phi(t,\tau)X_0=\Phi(t,\tau_1)X_1=\Phi(t,\tau_1)\,\Phi(\tau_1,\tau)X_0\,,
\qquad \forall \tau_1\in (\tau, t)\,.
\end{equation*}
Furthermore, as shown above, $\hat{u}(t,\tau;X_0)= \hat{u}(t,\tau_1;\Phi(\tau_1,\tau)X_0)$ holds as well.
\\
The proof of the statement S2. is concluded.



\section{An ensemble of optimal cost operators, the feedback formula. 
Statements S3.~and S4.~of Theorem~\ref{t:main}}
\label{s:feedback}
While the existence of a unique optimal control for the optimization problem \eqref{e:family-cauchy}-\eqref{e:family-cost} follows by a standard argument,
we aim at providing a representation of the said optimal control in {\em feedback}
form ({\em open- vs closed-loop} control).
In order to achieve the intended goal, we will establish a first representation formula, that is \eqref{e:pre-feedback} below.
This can be done rather easily, starting from the formula \eqref{e:feedback_0} that connects the optimal control $\hat{u}$ to the optimal state $\hat{w}$ (a consequence of the optimality condition), and next taking advantage of the transition property satisfied by the optimal pair.
This analysis is carried out in Section~\ref{ss:pre-feedback}.

To single out within the said representation certain operators ($P_0(\tau)$ and $P_1(t,\tau)$) that also occur in the quadratic form that yields the optimal cost,
derived in section~\ref{ss:riccati-ops}, further computations are necessitated.
This is eventually explored and achieved in section~\ref{ss:challenge}.


\subsection{A first representation of the optimal control in terms of the optimal state}
\label{ss:pre-feedback}
The relation between the optimal control and the initial state \eqref{e:from-optimality} following from the optimality condition can be rendered more explicit, via a representation formula that will play a crucial role in the sequel.
Indeed, we have
\begin{equation*}
\begin{split}
\hat{u}(t)&=-\big[\Lambda_\tau^{-1} N_\tau X_0\big](t)
= -\Lambda_\tau^{-1}L_\tau^*Q\Big[F(\cdot,\tau)\xi_0+\int_0^\tau M(\cdot,\sigma,\tau)\xi(\sigma)\,d\sigma\Big](t)
\\[1mm]
& = -\left[\left\langle 
\begin{pmatrix}
\Lambda_\tau^{-1}L_\tau^*QF(\cdot,\tau)
\\[1mm]
\Lambda_\tau^{-1}L_\tau^*QM(\cdot,:,\tau)
\end{pmatrix}
\begin{pmatrix}
\xi_0
\\
\xi(:)
\end{pmatrix}
\right\rangle_{Y_\tau}\right](t)\,.
\end{split} 
\end{equation*} 
Thus, by introducing the notation
\begin{subequations} \label{e:def-psi12}
\begin{align}
\Psi_1(t,\tau) &:= -\Big[\Lambda_\tau^{-1}L_\tau^*QF(\cdot,\tau)\Big](t)\,,
\label{e:def-psi1}
\\
\Psi_2(t,\sigma,\tau) &:= -\Big[\Lambda_\tau^{-1}L_\tau^*QM(\cdot,\sigma,\tau)\Big](t)\,,
\label{e:def-psi2}
\end{align}
\end{subequations}
we finally attain  
\begin{equation} \label{e:u-intermsof-Psi}
\hat{u}(t)=\Psi_1(t,\tau)\xi_0+\int_0^\tau \Psi_2(t,\sigma,\tau)\xi(\sigma)\,d\sigma\,.
\end{equation} 


\begin{lemma}
Let $(\hat{u}(\cdot,\tau;X_0),\hat{w}(\cdot,\tau;X_0))$ ($(\hat{w},\hat{u})$, in short)
be the optimal pair for the minimization problem \eqref{e:family-cauchy}-\eqref{e:family-cost},
with initial state $X_0$ (as in \eqref{e:X_0}).
Then,
\begin{equation} \label{e:pre-feedback}
\begin{split}
\hat{u}(t,\tau;X_0)&=-\int_\tau^T B^*F(\sigma,t)^* Q Z_1(\sigma,t)\hat{w}(t,\tau;X_0)\,d\sigma
\\ 
& \qquad -\int_\tau^T B^*F(\sigma,t)^* Q\int_0^t Z_2(\sigma,s,t)\hat{y}(s)\,ds\,d\sigma\,,
\end{split}
\end{equation}
where $\hat{y}(\cdot)$ is given by \eqref{e:evolution-map},
and having set for $\tau\le s\le t$
\begin{subequations} \label{e:def-Z12}
\begin{align}
Z_1(t,\tau) &:= F(t,\tau)+\int_\tau^t F(t,\sigma) B\Psi_1(\sigma,\tau)\,d\sigma\,,
\\
Z_2(t,s,\tau) &:= M(t,s,\tau)+\int_\tau^t F(t,\sigma) B\Psi_2(\sigma,s,\tau)\,d\sigma
\end{align}
\end{subequations}
(with $\Psi_i$, $i=1,2,$ as in \eqref{e:def-psi12}).

\end{lemma}


\begin{proof}
We rewrite the representation formula \eqref{e:represent} for the mild solutions to the integro-differential problem, with $\hat{w}$ and $\hat{u}$ in place of $w$ and $u$, respectively:
\begin{equation*}
\hat{w}(t,\tau;X_0)= F(t,\tau)\xi_0 + \int_0^\tau M(t,\sigma,\tau) \xi(\sigma)\,d\sigma
+ \int_\tau^t F(t,\sigma)B\hat{u}(\sigma)\,d\sigma\,.
\end{equation*}
Next, taking into account the expression \eqref{e:u-intermsof-Psi} of the optimal control
$\hat{u}$, we obtain
\begin{equation*}
\begin{split}
\hat{w}(t,\tau;X_0)
&= F(t,\tau)\xi_0 + \int_0^\tau M(t,\sigma,\tau) \xi(\sigma)\,d\sigma
\\
& \qquad\qquad
+ \int_\tau^t F(t,\sigma)B \Big[\Psi_1(\sigma,\tau)\xi_0+\int_0^\tau \Psi_2(\sigma,s,\tau)\xi(s)\,ds\Big]\,d\sigma\,,
\end{split}
\end{equation*}
which becomes
\begin{equation} \label{e:w-intermsof-Z}
\hat{w}(t,\tau;X_0)=Z_1(t,\tau)\xi_0 + \int_0^\tau Z_2(t,s,\tau) \xi(s)\,ds\,,
\end{equation}
by making use of the novel functions $Z_1(t,\tau)$ and $Z_2(t,s,\tau)$ defined in \eqref{e:def-Z12}.
(Just note the similar structure of the representations \eqref{e:u-intermsof-Psi} and \eqref{e:w-intermsof-Z}.)

Thus, recalling the expression \eqref{e:feedback_0} of the optimal control $\hat{u}(t;\tau,X_0)$ in terms of the optimal state (that follows from the optimality condition), we arrive at 
\begin{equation*}
\begin{split}
\hat{u}(t,\tau;X_0) &= - \int_t^T B^*F(\sigma,t)^*Q \hat{w}(\sigma,\tau;X_0)\,d\sigma
\\
& =- \int_t^T B^*F(\sigma,t)^*Q \hat{w}(\sigma,t;\Phi(t,\tau) X_0)\,d\sigma
\\
& =- \int_t^T B^*F(\sigma,t)^*Q Z_1(\sigma,t) \hat{w}(t,\tau;X_0)\,d\sigma
\\
& \qquad - \int_t^T B^*F(\sigma,t)^*Q \int_0^t Z_2(\sigma,s,t) \hat{y}(s)\,ds\,d\sigma\,,
\end{split}
\end{equation*}
which establishes \eqref{e:pre-feedback}.
We just observe that in the second equality we have used the transition property fulfilled by the optimal state, while in the third one we appealed to the representation \eqref{e:w-intermsof-Z} of the optimal state $\hat{w}(\sigma,t;\Phi(t,\tau)X_0)$ in terms of the initial datum $\Phi(t,\tau)X_0$, which reads as
\begin{equation*}
\begin{split}
&\hat{w}(\sigma,t;\Phi(t,\tau) X_0)= Z_1(\sigma,t) \hat{w}(t,\tau;X_0)
+\int_0^t Z_2(\sigma,s,t) \hat{y}(s)\,ds
\\
&\qquad = Z_1(\sigma,t) \hat{w}(t,\tau;X_0) +\int_0^\tau Z_2(\sigma,s,t) \xi(s)\,ds
+ \int_\tau^t Z_2(\sigma,s,t) \hat{w}(s,\tau;X_0)\,ds\,.
\end{split}
\end{equation*}
\end{proof}


\subsection{The optimal cost operators. Proof of statement S3.} \label{ss:riccati-ops}
We now use the formulas \eqref{e:w-intermsof-Z} and \eqref{e:u-intermsof-Psi} 
-- with $\Psi_i$, $Z_i$, $i=1,2,$ defined in \eqref{e:def-psi12} and \eqref{e:def-Z12},
respectively -- to compute the optimal value of the cost functional \eqref{e:family-cost}:
\begin{equation}
\begin{split}
J_\tau(u,X_0)&=\int_\tau^T \big[\langle Q\hat{w}(\sigma,\tau;X_0),\hat{w}(\sigma,\tau;X_0\rangle_H 
+ \|\hat{u}(\sigma,\tau;X_0)\|_U^2\big]\,d\sigma\
\\
& = \int_\tau^T \Big\{\Big\|Q^{1/2}\Big[Z_1(\sigma,\tau)\xi_0 + \int_0^\tau Z_2(\sigma,s,\tau) \xi(s)\,ds\Big]\Big\|^2_H
\\
& \myspace
+ \Big\|\Psi_1(\sigma,\tau)\xi_0+\int_0^\tau \Psi_2(\sigma,s,\tau)\xi(s)\,ds\Big\|_U^2\Big\}\,d\sigma\,.
\end{split}
\end{equation}
We replicate the formula, with $\Phi(\tau_1,\tau)X_0$ and $\hat{y}(\cdot)$ as in \eqref{e:opt_X_1} in place of $X_0$ and $\xi(\cdot)$, respectively, to find 
\begin{equation}
\begin{split}
& J_{\tau_1}(\hat{u},\Phi(\tau_1,\tau)X_0)
= \int_{\tau_1}^T \Big\{\Big\|Q^{1/2}\Big[Z_1(\sigma,\tau_1)\hat{w}(\tau_1,\tau;X_0) 
+ \int_0^{\tau_1}Z_2(\sigma,s,\tau_1) \hat{y}(s)\,ds\Big]\Big\|^2_H 
\\
& \myspace
+ \Big\|\Psi_1(\sigma,\tau_1)\hat{w}(\tau_1,\tau;X_0) +\int_0^{\tau_1} \Psi_2(\sigma,s,\tau_1)\hat{y}(s)\,ds\Big\|_U^2\Big\}\,d\sigma\,.
\end{split}
\end{equation}
Computing the squares, we find
\begin{equation*}
\begin{split}
& J_{\tau_1}(\hat{u},\Phi(\tau_1,\tau)X_0)
= \int_{\tau_1}^T \Big\{\Big\|Q^{1/2}Z_1(\sigma,\tau_1)\hat{w}(\tau_1,\tau;X_0)\Big\|_H^2
\\
& \myspace
+ 2\text{Re}\, \Big\langle Q Z_1(\sigma,\tau_1)\hat{w}(\tau_1,\tau;X_0),\int_0^{\tau_1}Z_2(\sigma,s,\tau_1) \hat{y}(s)\,ds\Big\rangle_H  
\\
& \myspace
+ \Big\|Q^{1/2}\int_0^{\tau_1}Z_2(\sigma,s,\tau_1) \hat{y}(s)\,ds\Big\|_H^2
+ \big\|\Psi_1(\sigma,\tau_1)\hat{w}(\tau_1,\tau;X_0)\Big\|_U^2 
\\
& \myspace
+  2\text{Re}\,\Big\langle \Psi_1(\sigma,\tau_1)\hat{w}(\tau_1,\tau;X_0),\int_0^{\tau_1} \Psi_2(\sigma,s,\tau_1)\hat{y}(s)\,ds\Big\rangle_U
\\
& \myspace
+ \Big\|\int_0^{\tau_1} \Psi_2(\sigma,s,\tau_1)\hat{y}(s)\,ds\Big\|_U^2\Big\}\,d\sigma\,.
\end{split}
\end{equation*}
Suitably rearranging the summands and setting 
\begin{subequations} \label{e:riccati-ops}
\begin{align}
P_0(t)&=\int_t^T \big[\Psi_1(p,t)^*\Psi_1(p,t)+Z_1(p,t)^*Q Z_1(p,t)\big]\,dp\,,
\label{e:P_0}
\\[1mm]
P_1(t,s)&=\int_t^T \big[\Psi_1(p,t)^*\Psi_2(p,s,t)+Z_1(p,t)^*QZ_2(p,s,t)\big]\,dp
\label{e:P_1}
\\[1mm]
P_2(t,s,q)&=\int_t^T \big[\Psi_2(p,q,t)^*\Psi_2(p,s,t)+Z_2(p,q,t)^*QZ_2(p,s,t)\big]\,dp\,,
\label{e:P_2}
\end{align}
\end{subequations}
we obtain
\begin{equation*} 
\begin{split}
J_{\tau_1}(\hat{u},\Phi(\tau_1,\tau)X_0)
& = 
\big\langle P_0(\tau_1)\hat{w}(\tau_1,\tau;X_0),\hat{w}(\tau_1,\tau;X_0)\big\rangle_H 
\\
& \qquad + 2\text{Re}\,\int_0^{\tau_1} 
\big\langle  P_1(\tau_1,s) \hat{y}(s),\hat{w}(\tau_1,\tau;X_0)\big\rangle_H \,ds
\\
& \qquad
+ \int_0^{\tau_1} \int_0^{\tau_1} \big\langle P_2(\tau_1,s,q) \hat{y}(s),\hat{y}(q)\big\rangle\,ds\,dq\,,
\end{split}
\end{equation*}
which is rewritten as
\begin{equation} \label{e:optimal-cost_2}
J_{\tau_1}(\hat{u},\Phi(\tau_1,\tau)X_0)
= \big\langle P(\tau_1) \Phi(\tau_1,\tau)X_0, \Phi(\tau_1,\tau)X_0\big\rangle_{Y_{\tau_1}}\,, 
\end{equation}
with $P(\cdot)$ the operator defined by \eqref{e:big-riccati}.

If we set now $\tau_1=\tau$ in \eqref{e:optimal-cost_2}, we finally establish the sought representation \eqref{e:optimal-cost_1} of the optimal cost as a quadratic form on the state space $Y_\tau$.

\begin{remark}
\begin{rm}
The matrix operator $P(\cdot)$ is the `want-to-be' Riccati operator of the optimal control problem,
namely, a candidate solution to an appropriate Riccati equation in a
dense subspace of the space $Y_\tau$.
This constitutes a first part of the statement of Theorem~\ref{t:big-riccati}. 
It will be proved rigorously in Section~\ref{s:uniqueness}, on the basis of the assertion S5. of Theorem~\ref{t:main} and in the light of Proposition~\ref{p:DRE_equiv}.
\end{rm}
\end{remark}

We note that the properties 
\begin{equation}\label{e:P_i_pos_self}
P_0(t)=P_0(t)^*\ge 0,\quad P_2(t,s,q)=P_2(t,q,s)=P_2(t,s,q)^*\ge 0\,.
\end{equation}
are intrinsic to the respective definitions \eqref{e:P_0} and \eqref{e:P_2} of the operators
$P_0$ and $P_2$.

It is worth emphasizing at the outset and explicitly the basic regularity properties of the optimal cost operators $P_i$ ($i=1,2,3$).


\begin{proposition} \label{p:cont_P} 
The operators $P_0(t)$, $P_1(t,s)$, $P_2(t,s,q)$ defined in \eqref{e:riccati-ops}
possess the following regularity:
\begin{itemize}

\item 
for every $t\in [0,T]$ and $s,q\in [0,t]$, $P_0(t)$, $P_1(t,s)$, $P_2(t,s,q)$ belong to $\cL(H)$,
with respective norms bounded by some constant $c$;

\item 
$P_0(t)$, $P_1(t,s)$, $P_2(t,s,q)$ are continuous functions (with respect to their variables),
with values in $\cL(H)$. 

\end{itemize}
\end{proposition}

\begin{proof}
The assertions follow given the respective definitions of $P_0(t)$, $P_1(t,s)$ and $P_2(t,s,q)$, 
after a careful analysis of the continuity properties of the operators $\Psi_i$, $i=1,2$ 
(see \eqref{e:def-psi12}), $Z_i$, $i=1,2$ (see \eqref{e:def-Z12}, $F$, $M$ (and $G$, $R$, $\mu$;
see \eqref{e:F}, \eqref{e:M}, \eqref{e:G}, \eqref{e:mu}, \eqref{e:resolvent-op}). 
We omit the details.
\end{proof}


\subsection{The feedback representation of the optimal control. Proof of statement S4.}\label{ss:challenge}
In the previous subsections we derived 

\begin{itemize}

\item
a first (pointwise in time) representation of the optimal control in terms of the optimal state, that is \eqref{e:pre-feedback};

\item
the representation of the optimal cost $J_\tau(\hat{u})$ as the quadratic form 
\eqref{e:optimal-cost_1} in the space $Y_\tau$. 

\end{itemize}
However, differently from the case of the LQ problem for memoryless equations, one cannot single out readily the presence of the operator $P(\cdot)\in \cL(Y_\tau)$ associated with this quadratic form within the formula \eqref{e:pre-feedback}.
To disclose the said presence, the following result is critical.


\begin{lemma}[Key Lemma] \label{l:key-lemma}
With the functions $\Psi_1(\sigma,t)$ and $\Psi_2(\sigma,s,t)$ defined in \eqref{e:def-psi12}, and $Z_1(\sigma,t)$ and $Z_2(\sigma,s,t)$ defined in \eqref{e:def-Z12}, the following identities hold true:
\begin{equation} \label{e:critical}
\begin{split}
\int_t^T F(\sigma,t)^*QZ_1(\sigma,t)\,d\sigma
&= \int_t^T \big[Z_1(\sigma,t)^*QZ_1(\sigma,t) +\Psi_1(\sigma,t)^*\Psi_1(\sigma,t)\big]\,d\sigma\,,
\\[1mm]
\int_t^T F(\sigma,t)^*QZ_2(\sigma,s,t)\,d\sigma
& = \int_t^T \big[Z_1(\sigma,t)^*QZ_2(\sigma,s,t)+\Psi_1(\sigma,t)^*\Psi_2(\sigma,s,t)\big]\,d\sigma\,.
\\[1mm]
\int_t^T M(p,q,t)^*QZ_2(p,s,t)\,dp
&= \int_t^T \big[Z_2(p,q,t)^*QZ_2(p,s,t)+\Psi_2(p,q,t)^*\Psi_2(p,s,t)\big]\,dp\,.
\end{split}
\end{equation}
As a consequence, the optimal cost operators $P_i$, $i=1,2,3$ in \eqref{e:riccati-ops}
admit the following respective representations, as well:
\begin{subequations} \label{e:riccati-ops_2}
\begin{align}
P_0(t)&=\int_t^T F(\sigma,t)^*QZ_1(\sigma,t)\,d\sigma
\label{e:P_0_v2}
\\[1mm]
P_1(t,s)&=\int_t^T F(\sigma,t)^*QZ_2(\sigma,s,t)\,d\sigma
\label{e:P_1_v2}
\\[1mm]
P_2(t,s,q)&=\int_t^T M(p,q,t)^*QZ_2(p,s,t)\,dp\,.
\label{e:P_2_v2}
\end{align}
\end{subequations}

\end{lemma}

\begin{proof}
(i) In order to establish the first one of the identities \eqref{e:critical},
we take the difference between its left and right hand sides, that is
\begin{equation*}
\begin{split}
& \int_t^T \big[F(\sigma,t)^*-Z_1(\sigma,t)^*\big]QZ_1(\sigma,t)\,d\sigma
-\int_t^T \Psi_1(\sigma,t)^*\Psi_1(\sigma,t)\,d\sigma
\\
& \quad = - \int_t^T\Big[\int_t^\sigma \Psi_1(q,t)^*B^*F(\sigma,q)^*\,dq\Big]
Q Z_1(\sigma,t)\,d\sigma
-\int_t^T \Psi_1(\sigma,t)^*\Psi_1(\sigma,t)\,d\sigma
\\
& \quad = - \int_t^T\int_q^T \Psi_1(q,t)^*B^*F(\sigma,q)^*Q Z_1(\sigma,t)\,d\sigma dq
-\int_t^T \Psi_1(q,t)^*\Psi_1(q,t)\,dq
\\
& \quad = - \int_t^T \Psi_1(q,t)^*\Big[\big[L_t^* Q Z_1(\cdot,t)\big](q)+\Psi_1(q,t)\Big]\,dq\,,
\end{split}
\end{equation*}
where we made use of the definitions of $Z_1$ (see \eqref{e:def-Z12}) as well
as the one of $L_\tau^*$ (see \eqref{e:L_tau-star}).

With the last expression as a starting point we substitute once more the expression of $Z_1$ and
move on with the computations, to find
\begin{equation*}
\begin{split}
& \int_t^T \big[F(\sigma,t)^*-Z_1(\sigma,t)^*\big]QZ_1(\sigma,t)\,d\sigma
-\int_t^T \Psi_1(\sigma,t)^*\Psi_1(\sigma,t)\,d\sigma
\\
& \quad = 
- \int_t^T \Psi_1(q,t)^*\big[L_t^* Q [F(\cdot,t)+L_t\Psi_1(\cdot,t)](q)+\Psi_1(q,t)\big]\,dq
\\
& \quad = - \int_t^T \Psi_1(q,t)^*\big[\big(L_t^* Q F(\cdot,t)\big)(q)
+[\Lambda_t \Psi_1(\cdot,t)](q)\big]\,dq
\\
& \quad = - \int_t^T \Psi_1(q,t)^*\big[\big(L_t^* Q F(\cdot,t)\big)(q)
-\big(L_t^* Q F(\cdot,t)\big)(q)\big]\,dq\equiv 0\,,
\end{split}
\end{equation*}
as desired.
We note that in the last but one equality we recalled $L_t^* QL_t+I=:\Lambda_t$,
while in the last equality we utilized once again the definition of $\Psi_1$
in \eqref{e:def-psi12}.

\smallskip
\noindent
(ii) We proceed in an analogous way, {\em mutatis mutandis}.
A first series of passages leads to
\begin{equation*}
\begin{split}
& \int_t^T \Big\{\big[F(\sigma,t)^*-Z_1(\sigma,t)^*\big]QZ_2(\sigma,s,t)
-\Psi_1(\sigma,t)^*\Psi_2(\sigma,s,t)\Big\}\,d\sigma
\\
& \quad = - \int_t^T\Big\{\Big[\int_t^\sigma \Psi_1(q,t)^*B^*F(\sigma,q)^*\,dq\Big]
Q Z_2(\sigma,s,t)+\Psi_1(\sigma,t)^*\Psi_2(\sigma,s,t)\Big\}\,d\sigma
\\
& \quad = - \int_t^T\int_q^T \Psi_1(q,t)^*B^*F(\sigma,q)^*Q Z_2(\sigma,s,t)\,d\sigma dq
-\int_t^T \Psi_1(q,t)^*\Psi_2(q,s,t)\,dq
\\
& \quad = - \int_t^T \Psi_1(q,t)^*\Big[\big(L_t^* Q Z_2(\cdot,s,t)\big)(q)+\Psi_2(q,s,t)\Big]\,dq\,.
\end{split}
\end{equation*}

Similarly as before, we utilize $\Lambda_t:=L_t^* QL_t+I$ and the definitions of $\Psi_2$
and $Z_2$, to find
\begin{equation*}
\begin{split}
& \int_t^T \big[F(\sigma,t)^*-Z_1(\sigma,t)^*\big]QZ_2(\sigma,s,t)\,d\sigma
-\int_t^T \Psi_1(\sigma,t)^*\Psi_2(\sigma,s,t)\,d\sigma
\\
& \quad = 
- \int_t^T \Psi_1(q,t)^*\big[L_t^* Q [M(\cdot,s,t)+L_t\Psi_2(\cdot,s,t)](q)+\Psi_2(q,s,t)\big]\,dq
\\
& \quad = - \int_t^T \Psi_1(q,t)^*\big[\big(L_t^* Q M(\cdot,s,t)\big)(q)
+[\Lambda_t \Psi_2(\cdot,s,t)](q)\big]\,dq
\\
& \quad = - \int_t^T \Psi_1(q,t)^*\big[\big(L_t^* Q M(\cdot,s,t)\big)(q)
-\big(L_t^* Q M(\cdot,s,t)\big)(q)\big]\,dq\equiv 0\,.
\end{split}
\end{equation*}

\smallskip
\noindent
(iii)
Once again, we take the difference
\begin{equation*}
\begin{split}
& \int_t^T \big[Z_2(p,q,t)^*QZ_2(p,s,t)+\Psi_2(p,q,t)^*\Psi_2(p,s,t)\big]\,dp
-\int_t^T M(p,q,t)^*QZ_2(p,s,t)\,dp
\\
& \quad 
= \int_t^T \Big[\int_t^p \Psi_2(\sigma,q,t)^*B^*F(p,\sigma)^*QZ_2(p,s,t)\,d\sigma
+\Psi_2(p,q,t)^*\Psi_2(p,s,t)\Big]\,dp
\\
& \quad 
= \int_t^T \Big[\int_\sigma^T \Psi_2(\sigma,q,t)^*B^*F(p,\sigma)^*QZ_2(p,s,t)\,dp\,d\sigma
+ \int_t^T \Psi_2(p,q,t)^*\Psi_2(p,s,t)\Big]\,dp\,.
\end{split}
\end{equation*}
Recall the definition of $L_\tau^*$ and move on with the computations to find
\begin{equation*}
\begin{split}
&  \int_t^T \big[Z_2(p,q,t)^*QZ_2(p,s,t)+\Psi_2(p,q,t)^*\Psi_2(p,s,t)\big]\,dp
-\int_t^T M(p,q,t)^*QZ_2(p,s,t)\,dp
\\
& \quad 
= \int_t^T \Psi_2(\sigma,q,t)^* \big[L_t^*QZ_2(\cdot,s,t)\big](\sigma)\,d\sigma
+ \int_t^T \Psi_2(p,q,t)^*\Psi_2(p,s,t)\Big]\,dp\,
\\
& \quad 
= \int_t^T \Psi_2(\sigma,q,t)^* \Big[L_t^*Q\big[M(\cdot,s,t)+L_t\Psi_2(\cdot,s,t)\big](\sigma)
+\Psi_2(\sigma,s,t)\Big]\,d\sigma
\\
& \quad 
= \int_t^T \Psi_2(\sigma,q,t)^* \Big[\big[L_t^*QM(\cdot,s,t)\big](\sigma)
+\big[\Lambda_t\Psi_2(\cdot,s,t)\big](\sigma)\Big]\,d\sigma
\\
& \quad 
= \int_t^T \Psi_2(\sigma,q,t)^* \Big[\big[L_t^*QM(\cdot,s,t)\big](\sigma)
-\big[L_t^*QM(\cdot,s,t)\big](\sigma)\Big]\,d\sigma\equiv 0\,,
\end{split}
\end{equation*}
as expected.
(In the last two equalities, we used again $L_\tau^* Q L_\tau+I=\Lambda_t$ first and 
the representation of $\Psi_2$ in \eqref{e:def-psi2} next.)

\smallskip
\noindent
(iv) Thus, the formulae \eqref{e:riccati-ops_2} follow combining the attained identities
\eqref{e:critical} with the original representations in \eqref{e:riccati-ops}.
\end{proof}

The reformulation \eqref{e:riccati-ops_2} of the optimal cost operators allows for a derivation
of the sought-after feedback representation of the optimal control. 


\begin{proposition} \label{p:closedloop}
Let $(\hat{u}(\cdot,\tau;X_0),\hat{w}(\cdot,\tau;X_0))$ ($(\hat{w},\hat{u})$, in short)
be the optimal pair for the minimization problem \eqref{e:family-cauchy}-\eqref{e:family-cost},with initial state $X_0$ (as in \eqref{e:X_0}).
Then, the optimal control $\hat{u}$ admits the feedback representation
\eqref{e:feedback}, that is
\begin{equation*}
\hat{u}(t,\tau;X_0)=-B^*P_0(t)\hat{w}(t;\tau,X_0)-\int_0^t B^*P_1(t,s)\hat{y}(s)\,ds\,, 
\quad \tau\le t\le T\,,
\end{equation*}
with $\hat{y}(\cdot)$ given by \eqref{e:evolution-map}.
\end{proposition}

\begin{proof}
The validity of the feedback formula \eqref{e:feedback} now follows readily from 
the former representation \eqref{e:pre-feedback} of the optimal control,
in the light of the first two identities in \eqref{e:critical}.
\end{proof}


\section{The coupled system of quadratic equations satisfied by the optimal cost operators.
Statement S5.~of Theorem~\ref{t:main}}
\label{s:riccati-eqns}
This section is entirely devoted to the proof of the crucial assertion {\bf S5.} of 
Theorem~\ref{t:main}, namely of the fundamental fact that the operators $P_0(t)$, $P_1(t,s)$, $P_2(t,s,q)$ -- arisen as the `building blocks' of the quadratic form defining the optimal cost \eqref{e:optimal-cost_1}, eventually shown to be given by \eqref{e:riccati-ops_2},
and with the former two of them entering the feedback formula \eqref{e:feedback} -- solve a system of three coupled partial differential equations, that is \eqref{e:DRE}.
We note that, as it is formulated, system \eqref{e:DRE} cannot be recast immediately as 
a quadratic {\em equation} in the augmented space $Y_\tau$,
since the scalar products in the second and third equations of \eqref{e:DRE} involve only elements
of $\cD(A)\subset H$ rather than also functions with values in $H$. 
This subtle issue will be addressed later in the proof of Theorem~\ref{t:big-riccati}.

\medskip
Let $P_0(t)$, $P_1(t,s)$ and $P_2(t,s,q)$ be the operators defined in \eqref{e:riccati-ops}.
The basic regularity of $P_i$, $i=0,1,2$, is stated in Proposition~\ref{p:cont_P}.
We will make use of the equivalent representations \eqref{e:riccati-ops_2} 
obtained in Lemma~\ref{l:key-lemma} throughout.
We additionally note that combining the definitions \eqref{e:def-Z12} of $Z_2(p,s,\tau)$ and \eqref{e:def-psi12} of $\Psi_2(\sigma,s,\tau)$, the following equivalent representation for $P_1(\tau,s)$ holds true,
\begin{equation} \label{e:P_1-bis}
P_1(\tau,s) = \int_\tau^T F(\sigma,\tau)^*Q
\big[M(\sigma,s,\tau)-[L_\tau\Lambda_\tau^{-1}L_\tau^*QM(\cdot,s,\tau)](\sigma)\,\big]\,d\sigma\,,
\end{equation}
which in turn yields, using \eqref{e:banali} in Appendix~\ref{a:appendix},
\begin{equation} \label{e:p1-tau2}
\begin{split}
P_1(\tau,\tau) &= \int_\tau^T F(\sigma,\tau)^*Q
\big[[I-L_\tau\Lambda_\tau^{-1}L_\tau^*]QM(\cdot,\tau,\tau)\big](\sigma)\,d\sigma
\\
& = -\int_\tau^T F(\sigma,\tau)^*Q
\big[[I-L_\tau\Lambda_\tau^{-1}L_\tau^*]QR(\cdot-\tau)\big](\sigma)\,d\sigma\,.
\end{split}
\end{equation}
This information will be essential later in the computations that follow. 

\smallskip
\noindent
{\em Proof of the statement S5. of Theorem~\ref{t:main}.} \hspace{2mm}
\\
{\bf i) Equation satisfied by $P_0$.}
We start from $P_0(\tau)$ as in \eqref{e:P_0_v2}:
in view of Propositions~\ref{p:3} and \ref{p:7}, the operators $F(\sigma,\tau)$ and $Z_1(\sigma,\tau)$ can be differentiated with respect to $\tau$, when acting on elements $x\in\cD(A)$.
For the sake of simplicity we neglect $x$ throughout.
Using also Lemma~\ref{l:0}, the derivative with respect to $\tau$ is computed to be initially
\begin{equation*}
\begin{split}
& 
P_0'(\tau)=\frac{d}{d\tau} \int_\tau^T F(\sigma,\tau)^*QZ_1(\sigma,\tau)\,d\sigma
\\
& \quad = -Q + \int_\tau^T \partial_\tau F(\sigma,\tau)^*QZ_1(\sigma,\tau)\,d\sigma
+ \int_\tau^T F(\sigma,\tau)^*Q\partial_\tau Z_1(\sigma,\tau)\,d\sigma
\\
& \quad = -Q + \int_\tau^T \big[-A^*F(\sigma,\tau)^*+R(\sigma-\tau)^*\big] QZ_1(\sigma,\tau)\,d\sigma
\\
&\myspace\qquad + \int_\tau^T F(\sigma,\tau)^*Q\big[\partial_\tau F(\sigma,\tau)-F(\sigma,\tau)B\Psi_1(\tau,\tau)
\big]\,d\sigma
\\
&\myspace\myspace
+ \int_\tau^T F(\sigma,\tau)^* Q\big[ L_\tau \partial_\tau\Psi_1(\cdot,\tau)\big](\sigma)\,d\sigma
\\
&\quad 
= -Q -A^*P_0(\tau) + \int_\tau^T R(\sigma-\tau)^*Q Z_1(\sigma,\tau)\,d\sigma
\\
&\qquad + \int_\tau^T F(\sigma,\tau)^*Q\big[-F(\sigma,\tau)A+R(\sigma-\tau)\big]\,d\sigma
\\
&\qquad
- \int_\tau^T F(\sigma,\tau)^* QF(\sigma,\tau)B\Psi_1(\tau,\tau)\,d\sigma
+ \int_\tau^T F(\sigma,\tau)^* Q\big[ L_\tau \partial_\tau\Psi_1(\cdot,\tau)\big](\sigma)\,d\sigma\,.
\end{split}
\end{equation*}
So far, we used simply the formula in \eqref{e:F-derivatives} for $\partial_\tau F(\sigma,\tau)$,
as found in Proposition~\ref{p:3}, and singled out the term $-A^*P_0(\tau)$.

It turns out to be useful to replace $F(\sigma,\tau)$, that occurs in the fourth summand (in the right hand side), with its expression following from the definition of $Z_1$ in \eqref{e:def-Z12}; 
at the same time, we use the formula \eqref{e:Psi1-derivative} for 
$\partial_\tau\Psi_1(\cdot,\tau)$ established in Proposition~\ref{p:4}
to rewrite the last summand.
Thus, we obtain
\begin{equation*}
\begin{split}
& P_0'(\tau)=
-Q -A^*P_0(\tau) + \int_\tau^T R(\sigma-\tau)^*Q Z_1(\sigma,\tau)\big]\,d\sigma
\\
& \myspace
+ \underbrace{\int_\tau^T F(\sigma,\tau)^*Q\Big[-Z_1(\sigma,\tau)A}_{-P_0(\tau)A}
+\big[L_\tau\Psi_1(\cdot,\tau)A\big](\sigma)+R(\sigma-\tau)\Big]\,d\sigma
\\
& \myspace - \int_\tau^T F(\sigma,\tau)^* QF(\sigma,\tau)B\Psi_1(\tau,\tau)\,d\sigma
\\
& \myspace 
-\int_\tau^T F(\sigma,\tau)^* Q
\big[L_\tau\Lambda_\tau^{-1}L_\tau^*QF_\tau(\cdot,\tau)\big](\sigma)\,d\sigma
\\
& \myspace
+\int_\tau^T F(\sigma,\tau)^* Q
\big[L_\tau \Lambda_\tau^{-1}L_\tau^*QF(\cdot,\tau)B\Psi_1(\tau,\tau)\big](\sigma)\,d\sigma\,.
\end{split}
\end{equation*}
We move on with the computations substituting the expression \eqref{e:def-psi1} of 
$\Psi_1(t,\tau)$ to find
\begin{equation*}
\begin{split}
P_0'(\tau) &
=-Q -A^*P_0(\tau) -P_0(\tau)A +\int_\tau^T R(\sigma-\tau)^*Q Z_1(\sigma,\tau)\big]\,d\sigma
\\
& \quad
-\int_\tau^T F(\sigma,\tau)^* Q\big[L_\tau \Lambda_\tau^{-1}L_\tau^*QF(\cdot,\tau) A\big](\sigma)\,d\sigma
+\int_\tau^T F(\sigma,\tau)^* Q R(\sigma-\tau)\,d\sigma
\\
& \quad -\int_\tau^T F(\sigma,\tau)^* Q
\big[I-L_\tau \Lambda_\tau^{-1}L_\tau^*QF(\cdot,\tau)B\Psi_1(\tau,\tau)\big](\sigma)\,d\sigma)
\\
& \quad
-\int_\tau^T F(\sigma,\tau)^* Q
\big[L_\tau\Lambda_\tau^{-1}L_\tau^*Q [-F(\cdot,\tau)A+R(\cdot-\tau)]\big](\sigma)\,d\sigma\,.
\end{split}
\end{equation*}
We note first that the fifth summand in the right hand side cancels with a portion of the eighth summand, then recall that $R(\sigma-\tau) = -M(\sigma,\tau,\tau)$;
thus, we get
\begin{equation*}
\begin{split}
& P_0'(\tau)= -Q -A^*P_0(\tau) -P_0(\tau)A 
\\
& \qquad
-\int_\tau^T M(\sigma,\tau,\tau)^*Q Z_1(\sigma,\tau)\,d\sigma
-\int_\tau^T F(\sigma,\tau)^* Q M(\sigma,\tau,\tau)\,d\sigma
\\
& \qquad
+\int_\tau^T F(\sigma,\tau)^* Q
\big[I-L_\tau \Lambda_\tau^{-1}L_\tau^*QF(\cdot,\tau)\big](\sigma)
B\big[\Lambda_\tau^{-1}L_\tau^*QF(\cdot,\tau)\big](\tau)\,d\sigma
\\
& \qquad
+\int_\tau^T F(\sigma,\tau)^* Q
\big[L_\tau\Lambda_\tau^{-1}L_\tau^*Q M(\cdot,\tau,\tau)]\big](\sigma)\,d\sigma\,.
\end{split}
\end{equation*}
Recalling the representations \eqref{e:def-psi2}, \eqref{e:def-Z12} and \eqref{e:P_1_v2} for 
$\Psi_2$, $Z_2$ and $P_1$, respectively, we disclose the presence of $P_1(\tau,\tau)$ as the sum
of the fifth and seventh summands.
Therefore,
\begin{equation*}
\begin{split}
& P_0'(\tau) = -Q -A^*P_0(\tau) -P_0(\tau)A 
\\
& \qquad
-\int_\tau^T M(\sigma,\tau,\tau)^*Q 
\Big[F(\sigma,\tau)-\big[L_\tau\Lambda_\tau^{-1}L_\tau^*Q F(\cdot,\tau)\big](\sigma)\Big]\,d\sigma
-P_1(\tau,\tau)
\\
& \qquad
+\int_\tau^T F(\sigma,\tau)^* Q
\Big[\big[I-L_\tau \Lambda_\tau^{-1}L_\tau^*Q\big]F(\cdot,\tau)
\big[B\Lambda_\tau^{-1}L_\tau^*QF(\cdot,\tau)\big](\tau)\Big](\sigma)\,d\sigma
\\
& \qquad
+\int_\tau^T F(\sigma,\tau)^* Q
\big[L_\tau\Lambda_\tau^{-1}L_\tau^*Q M(\cdot,\tau,\tau)\big](\sigma)\,d\sigma\,,
\end{split}
\end{equation*}
where the right hand side embeds the representation formula for $Z_1(t,\tau)$ in \eqref{e:useful}.

Taking into account that 
\begin{equation*}
\begin{split}
&\int_\tau^T M(\sigma,\tau,\tau)^*Q \big[I-L_\tau\Lambda_\tau^{-1}L_\tau^*Q F(\cdot,\tau)\big](\sigma)\,d\sigma
\\
& \myspace
=\int_\tau^T Z_2(\sigma,\tau,\tau)^* QF(\sigma,\tau)\,d\sigma=P_1(\tau,\tau)^*\,,
\end{split}
\end{equation*}
we achieve
\begin{equation*}
\begin{split}
& P_0'(\tau)
= -Q -A^*P_0(\tau) -P_0(\tau)A-P_1(\tau,\tau)-P_1(\tau,\tau)^*
\\
& \qquad 
+\underbrace{\int_\tau^T F(\sigma,\tau)^*Q
\Big[\big[I-L_\tau\Lambda_\tau^{-1}L_\tau^*\big]F(\cdot,\tau)
B \big[\Lambda_\tau^{-1}L_\tau^*QF(\cdot,\tau)\big](\tau)\Big](\sigma)\,d\sigma}_{T(\tau)}\,.
\end{split}
\end{equation*}

In order to complete this first part of the proof of the assertion S5. of Theorem~\ref{t:main}, it remains to unveil that the term $T(\tau)$ coincides exactly with the quadratic term 
$P_0(\tau)BB^*P_0(\tau)$.
We compute
\begin{equation*}
\begin{split}
& P_0(\tau)BB^*P_0(\tau) 
\\
& \quad
=\int_\tau^T\int_\tau^T F(\sigma,\tau)^* Q \big[F(\sigma,\tau)+[L_\tau\Psi_1(\cdot,\tau)](\sigma)\big]BB^*
\big[F(q,\tau)^* Q F(q,\tau)
\\
& \myspace\myspace
+\Psi_1(q,\tau)^*[L_\tau^*QF(\cdot,\tau)](q)\big]\,dq\,d\sigma
\\
& \quad
=\int_\tau^T \int_\tau^T F(\sigma,\tau)^* Q 
\big[F(\sigma,\tau)- [L_\tau\Lambda_\tau^{-1}L_\tau^*QF(\cdot,\tau)](\sigma)\big]
BB^*F(q,\tau)^*Q \big[F(q,\tau)-
\\
& \myspace\myspace
-[L_\tau\Lambda_\tau^{-1}L_\tau^*QF(\cdot,\tau)](q)\big]\,dq\,d\sigma
\\
& \quad
= \int_\tau^T F(\sigma,\tau)^*Q \big[[I-L_\tau\Lambda_\tau^{-1}L_\tau^*Q]F(\cdot,\tau)\big]
B\Big[L_\tau^*Q \big[[I-L_\tau\Lambda_\tau^{-1}L_\tau^*Q]F(\cdot,\tau)\big](\tau)\Big](\sigma)\,d\sigma
\\
& \quad
= \int_\tau^T F(\sigma,\tau)^*Q  
\big[[I-L_\tau\Lambda_\tau^{-1}L_\tau^*Q]F(\cdot,\tau)\big]
B\Big[\big[[L_\tau^* Q(\Lambda_\tau-I)\Lambda_\tau^{-1}L_\tau^*Q]F(\cdot,\tau)\big](\tau)\Big](\sigma)\,d\sigma
\\
& \quad
= \int_\tau^T F(\sigma,\tau)^*Q  
\big[[I-L_\tau\Lambda_\tau^{-1}L_\tau^*Q]F(\cdot,\tau)\big]
B\Big[\big[\Lambda_\tau^{-1}L_\tau^*Q F(\cdot,\tau)\big](\tau)\Big](\sigma)\,d\sigma\equiv T(\tau)\,,
\end{split}
\end{equation*}
where in the last but one equality we have rewritten $L_\tau^*QL_\tau$ as $\Lambda_\tau-I$, and in the last one we have seen that 
\begin{equation*}
L_\tau^* Q -(\Lambda_\tau-I)\Lambda_\tau^{-1}L_\tau^*Q
=L_\tau^* Q -L_\tau^*Q +\Lambda_\tau^{-1}L_\tau^*Q
=\Lambda_\tau^{-1}L_\tau^*Q\,.
\end{equation*}
The argument is complete.

\medskip

\noindent
{\bf ii) Equation satisfied by $P_1$.}
Achieving the equation satisfied by the operator $P_1(\tau,s)$ is a bit trickier.
Our starting point is the representation in \eqref{e:P_1_v2}; the operators $F(\sigma,\tau)$ and $Z_2(\sigma, s,\tau)$ are differentiable with respect to $\tau$ by Propositions \ref{p:3} and \ref{p:8}.
We compute the derivative (with respect to $\tau$) of $P_1(\tau,s)$:
\begin{equation*}
\begin{split}
& \partial_\tau P_1(\tau,s)
=-F(\tau,\tau)^*QZ_2(\tau,s,\tau)+\int_\tau^T \partial_\tau F(\sigma,\tau)^*QZ_2(\sigma,s,\tau)\,d\sigma
\\
&\quad+\int_\tau^T F(\sigma,\tau)^*Q \partial_\tau Z_2(\sigma,s,\tau)\,d\sigma
\\
& \ =\int_\tau^T [-F(\sigma,\tau)A+R(\sigma-\tau)]^* Q Z_2(\sigma,s,\tau)\,d\sigma
\\
&\quad +\int_\tau^T F(\sigma,\tau)^*Q \big[\partial_\tau M(\sigma,s,\tau) 
- F(\sigma,\tau)B\Psi_2(\tau,s,\tau)
+[L_\tau\partial_\tau \Psi_2(\cdot,s,\tau)](\sigma)\big]\,d\sigma\,,
\end{split}
\end{equation*}
where we used that $Z_2(\tau,s,\tau)=0$ (see \eqref{e:Z2-derivative} of Proposition~\ref{p:8}), and the representation \eqref{e:def-Z12} of $Z_2(\sigma,s,\tau)$. 

Then, we see that
\begin{equation*}
\begin{split}
&\partial_\tau P_1(\tau,s)
=-A^*P_1(\tau,s) +\int_\tau^T R(\sigma-\tau)^* Q Z_2(\sigma,s,\tau)\,d\sigma
\\
&\ \, 
-\int_\tau^T F(\sigma,\tau)^*Q F(\sigma,\tau)K(\tau -s)\,d\sigma  
+\int_\tau^T  F(\sigma,\tau)^*QF(\sigma,\tau)B
\big[\Lambda_\tau^{-1}L_\tau^*QM(\cdot,s,\tau)\big](\tau)\,d\sigma
\\
&\ \, - \int_\tau^T F(\sigma,\tau)^*Q 
L_\tau \Big[
\big[\Lambda_\tau^{-1}L_\tau^*Q \partial_\tau M(\cdot,s,\tau)\big] 
-\Lambda_\tau^{-1}L_\tau^*QF(\cdot,\tau) B\Psi_2(\tau,s,\tau)\Big](\sigma)\,d\sigma\,;
\end{split}
\end{equation*}
namely,
\begin{equation} \label{e:alfa}
\partial_\tau P_1(\tau,s)= \sum_{i=1}^6 S_i\,,
\end{equation}
where it is immediately seen that 
\begin{subequations} \label{e:beta}
\begin{align}
& S_1= -A^*P_1(\tau,s)\,,
\\
& S_2:= \int_\tau^T R(\sigma-\tau)^* Q Z_2(\sigma,s,\tau)\,d\sigma= -P_2(\tau,s,\tau)
\end{align}
\end{subequations} 
(the latter equality follows recalling $M(\sigma,\tau,\tau)=-R(\sigma-\tau)$ from \eqref{e:banali} in the Appendix~\ref{a:appendix}).

As for the summand $S_3$, we get
\begin{equation}\label{e:gamma}
\begin{split}
S_3 &=-\int_\tau^T F(\sigma,\tau)^*Q F(\sigma,\tau)\,d\sigma\,K(\tau-s)
\\
& = -\int_\tau^T F(\sigma,\tau)^*Q \big\{Z_1(\sigma,\tau)-[L_\tau\Psi_1(\cdot,\tau)](\sigma)\big\}\,d\sigma\,K(\tau-s)
\\
& =-P_0(\tau)K(\tau-s)
+\int_\tau^T F(\sigma,\tau)^*Q 
\Big[L_\tau \big[-\Lambda_\tau^{-1}L_\tau^*QF(\cdot,\tau)\big]\Big](\sigma)\,d\sigma\,K(\tau-s)
\\
&
=: S_{31}+ S_{32}\,.
\end{split}
\end{equation}
Thus, we note that 
\begin{equation*}
\begin{split}
S_5 &:= -\int_\tau^T F(\sigma,\tau)^*Q L_\tau \big[\Lambda_\tau^{-1}L_\tau^*Q \partial_\tau M(\cdot,s,\tau)\big](\sigma)\,d\sigma
\\
& = \int_\tau^T F(\sigma,\tau)^*Q L_\tau \big[\Lambda_\tau^{-1}L_\tau^*Q 
F(\cdot,\tau)K(\tau-s)\big](\sigma)\,d\sigma
\equiv - S_{32}\,,
\end{split}
\end{equation*}
so that 
\begin{equation}\label{e:delta}
S_{32} + S_5= S_{32} - S_{32}=0\,.
\end{equation}

It remains to consider $S_4+S_6$, where
\begin{align*}
& S_4= \int_\tau^T  F(\sigma,\tau)^*QF(\sigma,\tau)B
\big[\Lambda_\tau^{-1}L_\tau^*QM(\cdot,s,\tau)\big](\tau)\,d\sigma
\\
& S_6= \int_\tau^T  F(\sigma,\tau)^*Q \big[L_\tau \Lambda_\tau^{-1}L_\tau^*QF(\cdot,\tau) B\Psi_2(\tau,s,\tau)\big](\sigma)\,d\sigma
\\
& =\int_\tau^T  F(\sigma,\tau)^*Q L_\tau \Lambda_\tau^{-1}\big[L_\tau^*QF(\cdot,\tau) B
[-\Lambda_\tau^{-1}L_\tau^*Q M(\cdot,s,\tau)](\tau) \big](\sigma)\,d\sigma
\end{align*}
so that 
\begin{equation*}
S_4+S_6
=\Big[\int_\tau^T  F(\sigma,\tau)^*Q \big[I-L_\tau \Lambda_\tau^{-1}L_\tau^*Q\big]F(\cdot,\tau)(\sigma)\,d\sigma\Big]
B [\Lambda_\tau^{-1}L_\tau^*Q M(\cdot,s,\tau)](\tau)\,.
\end{equation*}


We prove now the following 

\begin{lemma} \label{l:4+6}
\begin{equation} \label{e:4+6}
S_4+S_6= P_0(\tau)BB^*P_1(\tau,s)\,.
\end{equation}
\end{lemma}

\begin{proof}
To compute the term $P_0(\tau)BB^*P_1(\tau,s)$ we appeal once more to the representations \eqref{e:P_0_v2} and \eqref{e:P_1_v2} (of $P_0(\tau)$ and $P_1(\tau,s)$, respectively), this time making use of the formulas \eqref{e:useful} for $Z_1(\sigma,t)$ and $Z_2(\sigma,s,t)$, respectively.
This gives
\begin{equation} \label{e:compute0}
\begin{split}
P_0(\tau)BB^*P_1(\tau,s)&= 
\int_\tau^T F(\sigma,\tau)^*Q \big\{\big[I-L_\tau\Lambda_\tau^{-1}L_\tau^*Q\big]
F(\cdot,\tau)\big\} (\sigma)\,d\sigma \,B
\\[1mm]
& \qquad
\times B^* \int_\tau^T F(\sigma,\tau)^*Q
\big\{\big[I-L_\tau\Lambda_\tau^{-1}L_\tau^*Q\big]M(\cdot,s,\tau)\big\}(q)\,dq
\\[1mm]
& =\int_\tau^T F(\sigma,\tau)^*Q \big\{\big[I-L_\tau\Lambda_\tau^{-1}L_\tau^*Q\big]
F(\cdot,\tau)\big\} (\sigma)\,d\sigma\,B 
\\[1mm]
& \qquad 
\times L_\tau^*Q \big\{\big[I-L_\tau\Lambda_\tau^{-1}L_\tau^*Q\big]M(\cdot,s,\tau)\big\}(\tau)\,,
\end{split}
\end{equation}
where the last equality is justified by the identification
\begin{equation*}
\begin{split}
& B^* \int_\tau^T F(\sigma,\tau)^*Q
\big\{\big[I-L_\tau\Lambda_\tau^{-1}L_\tau^*Q\big]M(\cdot,s,\tau)\big\}(q)\,dq
\\
& \qquad 
\equiv L_\tau^*Q
\big\{\big[I-L_\tau\Lambda_\tau^{-1}L_\tau^*Q\big]M(\cdot,s,\tau)\big\}(\tau)\,,
\end{split}
\end{equation*}
on the basis of the definition \eqref{e:L_tau-star} of $L_\tau^*$.

A key observation now is that the operator $L_\tau^*Q [I-L_\tau\Lambda_\tau^{-1}L_\tau^*Q]$
can be replaced by $\Lambda_\tau^{-1}L_\tau^*Q$: indeed, 
\begin{equation} \label{e:identity}
\begin{split}
L_\tau^*Q \big[I-L_\tau\Lambda_\tau^{-1}L_\tau^*Q\big]
&=L_\tau^*Q-[L_\tau^*QL_\tau] \Lambda_\tau^{-1}L_\tau^*Q
=L_\tau^*Q-[\Lambda_\tau-I]\Lambda_\tau^{-1}L_\tau^*Q
\\[1mm]
&
= \underbrace{L_\tau^*Q-L_\tau^*Q}_{\equiv 0}+\Lambda_\tau^{-1}L_\tau^*Q
=\Lambda_\tau^{-1}L_\tau^*Q\,,
\end{split}
\end{equation}
which inserted in \eqref{e:compute0} shows
\begin{equation*}
\begin{split}
&P_0(\tau)BB^*P_1(\tau,s)
=
\int_\tau^T F(\sigma,\tau)^*Q \big\{\big[I-L_\tau\Lambda_\tau^{-1}L_\tau^*Q\big]
F(\cdot,\tau)\big\} (\sigma)\,d\sigma\,B
\\
& \myspace\myspace 
\times\big[\Lambda_\tau^{-1}L_\tau^*Q M(\cdot,s,\tau)\big](\tau)\,. 
\end{split}
\end{equation*}
But this is nothing but \eqref{e:4+6}, which ends the proof of the lemma.

\end{proof}

Now we are ready to return to \eqref{e:alfa}, taking into account \eqref{e:beta}, \eqref{e:gamma}, 
\eqref{e:delta} and \eqref{e:4+6} of Lemma~\ref{l:4+6}, to achieve the desired conclusion that
$P_1(\tau,s)$ satisfies the equation
\begin{equation*}
\partial_\tau P_1(\tau,s)=-A^*P_1(\tau,s)-P_2(\tau,s,\tau)-P_0(\tau)K(\tau-s)+P_0(\tau)BB^*P_1(\tau,s)\,,
\end{equation*}
just like in \eqref{e:DRE}. 
Then, when acting on elements of $D(A)$, 
\begin{equation}\label{e:P_1*}
\partial_\tau P_1(\tau,q)^* =-P_1(\tau,q)^*A-P_2(\tau,\tau,q)- K(\tau-q)P_0(\tau)+P_1(\tau,q)^*BB^*P_0(\tau)
\end{equation}
as a consequence, 
since $P_0(\tau)$ and $P_2(t,s,q)$ are self-adjoint (see \eqref{e:P_i_pos_self}) and we have $P_2(t,q,s)=P_2(t,s,q)$.

\medskip

\noindent
{\bf iii) Equation satisfied by $P_2$.}
On the basis of the representation \eqref{e:P_2_v2} of $P_2(\tau,s,q)$, 
we recall that in view of Propositions \ref{p:5} and \ref{p:8} the operators $M(p,q,\tau)$ and $Z_2(p,s,\tau)$ are differentiable with respect to $\tau$ (still when acting on elements of $\cD(A)$). We begin the computation of the partial derivative $\partial_\tau P_2(\tau,s,q)$,
obtaining first
\begin{equation*}
\begin{split}
& \partial_\tau P_2(\tau,s,q)
=\cancel{-M(\tau,q,\tau)^*QZ_2(\tau,s,\tau)}
+\int_\tau^T \partial_\tau M(p,q,\tau)^*QZ_2(p,s,\tau)\,dp
\\
&\myspace +\int_\tau^T M(p,q,\tau)^*Q\partial_\tau Z_2(p,s,\tau)\,dp\,,
\end{split}
\end{equation*}
where the first summand cancels, as $M(\tau,q,\tau)=0$ (see Lemma~\ref{l:0}).
Then we appeal to Proposition~\ref{p:5} to find $\partial_\tau M(p,q,\tau)=-F(p,\tau)K(\tau-q)$, 
and to Proposition~\ref{p:8} to compute 
\begin{equation*}
\begin{split}
& \partial_\tau Z_2(p,s,\tau)=\partial_\tau M(p,s,\tau)x-F(p,\tau)B\Psi_2(\tau,s,\tau)x+\big[L_\tau\partial_\tau \Psi_2(\cdot,s,\tau)x\big](p)
\\
& \qquad 
= \partial_\tau M(p,s,\tau)-F(p,\tau)B\Psi_2(\tau,s,\tau)+\int_\tau^p F(p,\sigma)B\Psi_2(\sigma,s,\tau)\,d\sigma
\\
& \qquad 
= -F(p,\tau)K(\tau-s)-F(p,\tau)B\Psi_2(\tau,s,\tau)
\\
& \myspace
+\int_\tau^p F(p,\sigma)B\big[\Lambda_\tau^{-1}L_\tau^*QF(\cdot,\tau)K(\tau-s)\big](\sigma)\,d\sigma
\\
& \myspace +\int_\tau^p F(p,\sigma)B\big[\Lambda_\tau^{-1}L_\tau^*QF(\cdot,\tau)B\Psi_2(\tau,s,\tau)\big](\sigma)\,d\sigma\,.
\end{split}
\end{equation*}
Then
\begin{equation*}
\begin{split}
& \partial_\tau P_2(\tau,s,q)
= -\int_\tau^T K(\tau-q) F(p,\tau)^* Q M(p,s,\tau)\,dp
\\[1mm]
& \quad 
-\int_\tau^T K(\tau-q) F(p,\tau)^* Q[L_\tau\Psi_2(\cdot,s,\tau)](p)\,dp
\\[1mm]
& \quad
-\int_\tau^T M(p,q,\tau)^*QF(p,\tau)K(\tau-s)\,dp
- \int_\tau^T M(p,q,\tau)^*QF(p,\tau)B\Psi_2(\tau,s,\tau)\,dp
\\[1mm]
& \quad 
+ \int_\tau^T M(p,q,\tau)^*Q\int_\tau^p F(p,\sigma)B
\big[\Lambda_\tau^{-1}L_\tau^*QF(\cdot,\tau)K(\tau-s)\big](\sigma)\,d\sigma\,dp
\\[1mm]
& \quad 
+  \int_\tau^T M(p,q,\tau)^*Q\int_\tau^p F(p,\sigma)B
\big[\Lambda_\tau^{-1}L_\tau^*QF(\cdot,\tau)B\Psi_2(\tau,s,\tau)\big](\sigma)\,d\sigma\,dp
\\[1mm]
& \ 
= \sum_{i=1}^6 T_i\,.
\end{split}
\end{equation*}
First we observe that
\begin{equation}\label{e:one}
\begin{split}
T_1+T_2 &= -K(\tau-q)\Big[\int_\tau^T F(p,\tau)^* Q 
\big[M(p,s,\tau)+[L_\tau\Psi_2(\cdot,s,\tau)](p)\big]\,dp\Big]
\\
&=-K(\tau-q)P_1(\tau,s)\,;
\end{split}
\end{equation}
then, we compute
\begin{equation*}
\begin{split}
&T_3+T_5 =-\int_\tau^T M(p,q,\tau)^* Q F(p,\tau)\,dp\,K(\tau-s)
\\
& \myspace +\int_\tau^T M(p,q,\tau)^*Q
[L_\tau\Lambda_\tau^{-1} L_\tau^* QF(\cdot,\tau)](p)\,dp\,K(\tau-s)
\\[1mm]
& \quad 
= -\Big[\int_\tau^T M(p,q,\tau)^* Q F(p,\tau)\,dp 
-\int_\tau^T \Psi_2(p,q,\tau)^*[L_\tau^*QF(\cdot,\tau)](p)\,dp\Big]\,K(\tau-s)
\\[1mm]
&\quad 
=-\int_\tau^T Z_2(p,q,\tau)^* Q F(p,\tau)\,dp\,K(\tau-s)
\end{split}
\end{equation*}
(where in the last two equalities we used the second and fourth of the \eqref{e:4-adjoints}, respectively), to find
\begin{equation} \label{e:two}
T_3+T_5=-P_1(\tau,q)^*K(\tau-s)\,.
\end{equation}

It remains to pinpoint $T_4+T_6$. 
We have
\begin{equation*}
\begin{split}
&T_4+T_6 :=- \int_\tau^T M(p,q,\tau)^*QF(p,\tau)B\Psi_2(\tau,s,\tau)\,dp
\\
& \quad +\int_\tau^T M(p,q,\tau)^*Q\big[L_\tau
\Lambda_\tau^{-1}L_\tau^*QF(\cdot,\tau)B\Psi_2(\tau,s,\tau)\big](p)\,dp
\\[1mm]
& \; = \int_\tau^T M(p,q,\tau)^*Q\big[[I - L_\tau\Lambda_\tau^{-1}L_\tau^*Q]F(\cdot,\tau)B\Psi_2(\tau,s,\tau)\big](p)\,dp
\\
& \; = -\int_\tau^T M(p,q,\tau)^*Q\big[[I - L_\tau\Lambda_\tau^{-1}L_\tau^*Q]F(\cdot,\tau)B
[\Lambda_\tau^{-1}L_\tau^*QM(\cdot,s,\tau)](\tau)\big](p)\,dp
\\[1mm]
& \; = \int_\tau^T  M(p,q,\tau)^*Q\Big[\big[I - L_\tau \Lambda_\tau^{-1}L_\tau^*Q\big]
F(\cdot,\tau)B\big[ L_\tau^*Q[I-L_\tau \Lambda_\tau^{-1}L_\tau^*Q]M(\cdot,s,\tau)\big](\tau)\Big](p)dp,
\end{split}
\end{equation*}
where we used once again
$\Psi_2(\tau,s,\tau) = -\big[\Lambda_\tau^{-1}L_\tau^*QM(\cdot,s,\tau)\big](\tau)$ first,
while in the last equality we replaced $\Lambda_\tau^{-1}L_\tau^*Q$ by $L_\tau^*Q[I-L_\tau \Lambda_\tau^{-1}L_\tau^*Q]$ on the strenght of \eqref{e:identity}.

We insert the explicit meaning of the operator $L_\tau^*$ and carry on with the computations,
to find
\begin{equation} \label{e:three}
\begin{split}
&T_4+T_6 :=\int_\tau^T M(p,q,\tau)^*Q \big[[I-L_\tau \Lambda_\tau^{-1}
L_\tau^*Q\big]F(\cdot,\tau)(p)\big]\,dp
\\
& \myspace
\times BB^*\int_\tau^T F(\sigma,\tau)^*Q\big[[I-L_\tau \Lambda_\tau^{-1}L_\tau^*Q]M(\cdot,s,\tau)\big]
(\sigma)\,d\sigma
\\[1mm]
& \qquad 
=\int_\tau^T M(p,q,\tau)^*Q F(p,\tau)\,dp
- \int_\tau^T M(p,q,\tau)^*Q \big[L_\tau \Lambda_\tau^{-1}L_\tau^*QF(\cdot,\tau)\big](p)\,dp
\\
& \myspace
\times BB^*\int_\tau^T F(\sigma,\tau)^*Q
\big[M(\sigma,s,\tau)+[L_\tau \Psi_2(\cdot,s,\tau)](\sigma)\big]\,d\sigma
\\[1mm]
& \qquad 
=\int_\tau^T Z_2(p,q,\tau)^*Q F(p,\tau)\,dp\,B
B^*\int_\tau^T F(\sigma,\tau)^*QZ_2(\sigma,s,\tau)(\sigma)\,d\sigma
\\[1mm]
& \qquad 
= P_1(\tau,q)^*BB^*P_1(\tau,s)\,.
\end{split}
\end{equation}

Combining \eqref{e:three} with \eqref{e:one} and \eqref{e:two} we finally attain
\begin{equation*}
\partial_\tau P_2(\tau,s,q)
=-K(\tau-q)P_1(\tau,s)-P_1(\tau,s)^*K(\tau-s) +P_1(\tau,q)^*BB^*P_1(\tau,s)\,, 
\end{equation*}
as desired.
\qed


\section{Proof of Theorem~\ref{t:big-riccati}} \label{s:uniqueness}
\paragraph{\bf 1.~(Existence)}
To infer the {\em existence} of (at least) an operator solution to the Riccati-type equation \eqref{e:big-riccati}, it is natural to claim that it is the matrix operator $P(t)$ defined in \eqref{e:big-op} to solve \eqref{e:big-riccati} (in some dense subspace of $H\times L^2(0,t;H)$),
in the first place. 
This is indeed true.
However, there is a subtle analytical gap that needs to be dealt with, beforehand.
The task is fulfilled proving the following result in the first place.

\begin{proposition} \label{p:DRE_equiv} 
A triple $P_i$ ($i=0,1,2$) is a solution to the coupled system \eqref{e:DRE}
if and only if it solves the following one:
\begin{equation} \label{e:DREbis}
\begin{cases}
&\frac{d}{dt}\langle P_0(t)x,y\rangle_H +\langle P_0(t)x,Ay\rangle_H 
+ \langle Ax,P_0(t)y\rangle_H + \langle Qx,y\rangle_H  
\\[1mm] 
&\qquad\quad - \langle B^*P_0(t)x,B^*P_0(t)y\rangle_U 
+ \langle P_1(t,t)x,y\rangle_H +\langle x,P_1(t,t)y\rangle_H = 0 
\\[2mm]
&\frac{\partial}{\partial t} \langle P_1(t,\cdot)f(\cdot),y\rangle_{L^2(0,t;H)}
+ \langle P_1(t,\cdot)f(\cdot),Ay\rangle_{L^2(0,t;H)} 
\\[1mm]
& \qquad 
+\langle K(t-\cdot)f(\cdot),P_0(t)y\rangle_{L^2(0,t;H)} 
+ \langle P_2(t,\cdot,t)f(\cdot),y\rangle_{L^2(0,t;H)}
\\[1mm]  
&\qquad
- \langle B^*P_1(t,\cdot)f(\cdot),B^*P_0(t)y\rangle_{L^2(0,t;U)}=0
\\[2mm]
& \frac{\partial}{\partial t}\langle P_2(t,\cdot,:)f(\cdot),g(:)\rangle_{L^2((0,t)^2;H)} 
+\langle P_1(t,\cdot)f(\cdot),K(t-:)g(:)\rangle_{L^2((0,t)^2;H)}
\\[1mm]
& \qquad
+\langle K(t-\cdot)f(\cdot),P_1(t,:)g(:)\rangle_{L^2((0,t)^2;H)}
\\[1mm]
&\qquad  -\langle B^*P_1(t,\cdot)f(\cdot),B^*P_1(t,:)g(:)\rangle_{L^2((0,t)^2;U)}=0
\end{cases}
\end{equation}
(for all $t\in [0,T]$, $y\in \cD(A)$ and $f,g\in L^2(0,t;\cD(A))$). 
\end{proposition}

\begin{proof} 
Assune that a triple $P_i$ ($i=0,1,2$) solves the coupled system \eqref{e:DREbis} 
(for all $t\in (0,T]$, $y\in H$ and $f,g\in L^2(0,t;H)$). 
Given $\tau\in (0,t)$ and $x\in \cD(A)$, we choose 
\begin{equation*}
f(s)=
\begin{cases} x & \text{in $[0,\tau]$}
\\
0 & \text{in $(\tau,t]$}
\end{cases}
\end{equation*} 
in the second of the equations in \eqref{e:DREbis};
the scalar products bring about integrals that are performed on the interval $[0,\tau]$,
as a consequence.
If we write these explicitly, we see that the partial derivative $\partial_t$ can be moved
inside any integral.
Then, differentiating with respect to $\tau$, we obtain
\begin{equation*}
\begin{split}
&\frac{\partial}{\partial t} \langle P_1(t,\tau)x,y\rangle_H  + \langle P_1(t,\tau)x,Ay\rangle_H +\langle K(t-\tau)x,P_0(t)y\rangle_H
\\[1mm]
& \qquad 
+ \langle P_2(t,\tau,t)x,y\rangle_H - \langle B^*P_1(t,\tau)x,B^*P_0(t)y\rangle_U=0\,,
\end{split}
\end{equation*}
which is nothing but the second equation of \eqref{e:DRE}, as $\tau<t$ was (given, and yet) arbitrary.

Similarly, for any given $\tau, \sigma \in (0,t)$ and $x,y\in \cD(A)$, we choose $f$ as above
and set
\begin{equation*}
g(s)=\begin{cases} y & \text{in $[0,\sigma]$}
\\
0 & \text{in $(\sigma,t]$.}
\end{cases}
\end{equation*}
We insert these functions in the third equation of \eqref{e:DREbis} and
write explicitly the integrals; once again the derivative $\partial_t$ can be moved inside the integrals.
Differentiate with respect to $\tau$ first, and next with respect to $\sigma$, to find
\begin{equation*}
\begin{split}
&\frac{\partial}{\partial t}\langle P_2(t,\tau,\sigma)x,y\rangle_H 
+\langle P_1(t,\tau)x,K(t-\sigma)y\rangle_H
\\[1mm]
& \qquad
+\langle K(t-\tau)x,P_1(t,\sigma)y\rangle_H
 -\langle B^*P_1(t,\tau)x,B^*P_1(t,\sigma)y\rangle_U=0\,,
\end{split}
\end{equation*}
which is just the third equation in \eqref{e:DRE} (since both $\tau, \sigma$ are arbitrary).

\medskip
Suppose now, conversely, that a triple $P_i$ ($i=0,1,2$) solves the coupled system \eqref{e:DRE}
(for any $t\in [0,T]$ and $x,y\in \cD(A)$). 
With focus on the second equation, set $x=f(s)$ where $f\in C^0_0((0,T),\cD(A))$ and 
$f\equiv 0$ in $[t,T]$, and integrate over $[0,t]$. 
Then,
\begin{equation*}
\begin{split}
\int_0^t \frac{\partial}{\partial t} \langle P_1(t,s)f(s), y\rangle_H \,ds 
&= \frac{d}{dt}\left\langle \int_0^t P_1(t,s)f(s)\,ds,y\right\rangle_H
\\[1mm]
& =\frac{d}{dt}\langle P_1(t,\cdot)f(\cdot), y\rangle_{L^2(0,t;H)}\,,
\end{split}
\end{equation*}
where in the first equality we used that $f(t)=0$.
This means that the second equation in \eqref{e:DREbis} is valid for any 
$f\in C^0_0((0,t),\cD(A))$ and $y\in H$. 

\smallskip
If now $f\in L^2(0,t;\cD(A))$, we extend it on the entire interval $[0,T]$ by setting it to $0$
on $(t,T]$. 
Select a sequence $\{f_n\} \subset C^0_0((0,t),D(A))$ such that $f_n$ converges to $f$ in 
$L^2(0,t;\cD(A))$.
The second equation in \eqref{e:DREbis} is satisfied for any $f_n$; letting $n\to \infty$, all the terms on the right hand side converge to the corresponding ones, with $f$ in place of $f_n$.
In addition, it is verified readily that the convergence is uniform with respect to $t\in [0,T]$.
Thus, in the first hand side of the equation we also have 
\begin{equation}\label{e:der_conv}
\begin{split}
& \lim_{n\to \infty} \frac{d}{dt}\langle P_1(t,\cdot)f_n(\cdot), y\rangle_{L^2(0,t;H)} = -\langle P_1(t,\cdot)f(\cdot), Ay\rangle_{L^2(0,t;H)} 
\\[1mm]
& \qquad 
-\langle K(t-\cdot)f(\cdot),P_0(t)y\rangle_{L^2(0,t;H)} -\langle  P_2(t,s,t)f(s), y\rangle_{L^2(0,t;H)} 
\\[1mm]
&\qquad  
+\langle B^*P_1(t,\cdot)f(\cdot),B^*P_0(t)y\rangle_{L^2(0,t;U)}\,,
\end{split}
\end{equation}
uniformly with respect to $t\in [0,T]$.
Similarly,
\begin{equation}\label{e:fun_conv}
\lim_{n\to \infty}\langle P_1(t,\cdot)f_n(\cdot),y\rangle_{L^2(0,t;H)} 
= \langle P_1(t,\cdot)f(\cdot),y\rangle_{L^2(0,t;H)}  
\end{equation}
uniformly with respect to $t\in [0,T]$.

In view of \eqref{e:fun_conv} and \eqref{e:der_conv}, we obtain that 
$t \longmapsto\langle P_1(t,\cdot)f_n(\cdot),y\rangle_{L^2(0,t;H)}$ is differentiable 
and that the second equation in \eqref{e:DREbis} is satisfied (for any $f\in L^2(0,t;\cD(A))$ 
and any $y\in H$).

A similar argument proves that the third equation in \eqref{e:DREbis} is satisfied as well.
Just set $x=f(s)$, $y=g(q)$, with $f,g\in C^0_0((0,T),\cD(A))$ and $f\equiv g\equiv 0$
in $[t,T]$, and proceed ({\em mutatis mutandis}) as before.
This concludes the proof of the proposition.
\end{proof}


In light of Proposition~\ref{p:DRE_equiv}, the existence part of the proof of Theorem~\ref{t:big-riccati} now demands only a rewriting of the coupled system of four equations satisfied by the operators $P_0$, $P_1$, $P_1^*$, $P_2$ -- to wit, system \eqref{e:DREbis} complemented with the equation satisfied by $P_1^*$ following \eqref{e:P_1*} -- as the unique matrix equation 
\eqref{e:big-riccati}.
As such, it is omitted.

\smallskip

\paragraph{\bf 2.~(Uniqueness)}
In order to show uniqueness, it is necessary to make some preliminary remarks.
The operator $P(\tau)$ defined by \eqref{e:big-op} belongs to the space $\cL(Y_\tau)$ for
any $\tau\in (0,T]$, and hence $P(\cdot)$ belongs to the Banach space
\begin{equation*}
Z:=\big\{U(\cdot)\colon U(\tau)\in \cL(Y_\tau), \; \tau\in (0,T]\big\}, 
\quad \|U\|_Z= \sup_{\tau\in (0,T]} \|U(\tau)\|_{\cL(Y_\tau)}<\infty\,.
\end{equation*}
Of course, for $\tau<t$ we have $Y_t\subseteq Y_\tau$. 
On the other hand, given $g\in L^2(0,\tau;H)$, set 
\begin{equation*}
G(t) := \begin{cases}g(t) & \text{in $[0,\tau]$}
\\
0 & \text{in $(\tau,t]$;}
\end{cases}
\end{equation*}
then, for $\tau<t$ the map $(y,g)\longmapsto (y,G)$ from $Y_\tau$ to the set 
\begin{equation*}
E_{\tau,t}=\{(y,h)\in Y_t\colon \; y\in H, \; h(\cdot)\equiv 0 \;\, \text{in $(\tau,t]$}\}
\end{equation*}
is an isometry. 
We may therefore identify $Y_\tau$ with $E_{\tau,t}$. 
Hence, if $U\in Z$ and $(y,g)\in Y_\tau$ we may define $\|U(t)(y,g)\|_{Y_\tau}=\|U(t)(y,G)\|_{Y_t}$, and we have
\begin{equation*}
\|U(t)(y,g)\|_{Y_\tau}=\|U(t)(y,G)\|_{Y_t}\le \|U(t)\|_{\cL(Y_t)}\|(y,G)\|_{Y_t} = \|U(t)\|_{\cL(Y_t)}\|(y,G)\|_{Y_\tau}; 
\end{equation*}
in other words, we may say that $U(t)\in \cL(Y_\tau)$ for every $\tau<t$, with 
\begin{equation}\label{e:stimeY_t} 
\|U(t)\|_{\cL(Y_\tau)}\le \|U(t)\|_{\cL(Y_t)}, \qquad 0<\tau\le t\le T\,.
\end{equation}  

\smallskip 
We proceed by contradiction: let $U(\cdot)\in Z$ be another solution of \eqref{e:big-riccati}
(besides $P(\tau)$ defined by \eqref{e:big-op}), and set $V:=P-U$. 
Of course, we have as well
\begin{equation*}
U(\tau)=
\begin{pmatrix}
U_0(\tau) & U_1(\tau,\cdot)
\\
U_1(\tau,:)^* & U_2(\tau,\cdot,:)
\end{pmatrix}, 
\qquad
V(\tau)=
\begin{pmatrix}
V_0(\tau) & V_1(\tau,\cdot)
\\
V_1(\tau,:)^* & V_2(\tau,\cdot,:)
\end{pmatrix}, 
\quad \tau\in (0,T).
\end{equation*}
For given $\tau, t\in (0,T)$, with $\tau<t$, the difference operator $V(\cdot)$ satisfies in $[t,T)$ the differential equation
\begin{equation}\label{e:DRE_uni} 
\begin{split}
& \frac{d}{dr} V(r)+V(r)(\cA+\cK_1(r)+\cD_{1,r})+(\cA^*+\cK_2(r)+\cD_{2,r})V(r)
\\
& \myspace \myspace
- V(r)\cB\cB^*P(r) - U(r)\cB\cB^*V(r)=0,
\end{split}
\end{equation}
with $V(T)=P(T)-U(T)=0$.
We note that the operator $\cA$ is the infinitesimal generator of a $C_0$-semigroup 
$\{e^{s\cA}\}_{s\ge 0}$ in $Y_\tau$ (for any $\tau\in (0,T)$), explicity given by
\begin{equation*}
e^{s\cA} = 
\begin{pmatrix}
e^{sA} & 0
\\
0 & I
\end{pmatrix};
\end{equation*}
in addition, the estimate $\|e^{s\cA}\|_{\cL(Y_\tau)} \le C\,e^{\omega s}\wedge 1$, $s\ge 0$, holds true.
In particular, if $\omega\ge 0$, then $\|e^{s\cA}\|_{\cL(Y_\tau)} \le C\,e^{\omega s}$
(the computations below are similar in the case $\omega<0$).

The mild form of \eqref{e:DRE_uni} in $[t,T]$ is
\begin{equation*} 
\begin{split}
& V(r) = \int_r^T e^{(p-r)\cA^*}\Big[V(p)[\cK_1(p)+\cD_{1,p}]+[\cK_2(p)+\cD_{2,p}]V(p)
\\
& \myspace \myspace
-V(p)\cB\cB^*P(p) - U(p)\cB\cB^*V(p)\Big]e^{(p-r)\cA}\,dp\,.
\end{split}
\end{equation*}

It can be shown that for $r\ge t$ the map $r \longmapsto \|V(r)\|_{\cL(Y_t)}$ is lower semi-continuous; hence, it is a measurable function in $[t,T]$ (the proof is postponed; see Lemma~\ref{l:lsc} at the end of this section).
Therefore, we are allowed to estimate $\|V(r)\|_{\cL(Y_t)}$ for any given $r\in [t,T]$, to find 
\begin{equation*}
\begin{split} 
\|V(r)\|_{\cL(Y_t)} 
&\le \int_t^T C^2 e^{2\omega (p-r)}
\Big[ 2\|K\|_{L^2(0,T)}\|V(p)\|_{\cL(Y_t)} + 2\|V(p)\|_{\cL(Y_t)} 
\\
& \qquad + \big(\|P\|_Z+\|U\|_Z\big)\|BB^*\|_{\cL(H)} \|V(p)\|_{\cL(Y_t)}\Big]\,dp
\\
& \le C_1 \int_r^T  \|V(p)\|_{\cL(Y_t)}\,dp\,, 
\end{split}
\end{equation*}
where $C_1$ is a suitable positive constant, depending on $T$, $\|BB^*\|_{\cL(H)}$, $\|K\|_{L^2(0,T)}$, $\|P\|_Z$, $\|U\|_Z$, and we used \eqref{e:stimeY_t}. 
By the Gronwall Lemma it follows that 
$$\|V(r)\|_{\cL(Y_t)} =0$$ for $r\in[t,T]$. 
This means, in particular, $\|V(t)\|_{\cL(Y_t)} =0$. 
Since $t>\tau$ was given and yet arbitrary, it follows that 
\begin{equation*}
V_0(t)=0, \quad V_1(t,s)=0, \quad V_2(t,s,q)=0 \quad \forall s,q\in [0,\tau], 
\; \forall t\in [\tau,T],
\end{equation*}
i.e. $V(r)=0$ as an element of $\cL(Y_\tau)$. 
Since $\tau$ was also given arbitrarily, we attain $\|V(r)\|_{\cL(Y_r)}=0$ for all $r\in [0,T]$,
i.e. 
$V(\tau)=0$ as an element of $\cL(Y_\tau)$,
which means $U(\cdot)\equiv P(\cdot)$ in $[0,T]$.

Thus, this second part of the proof is complete once the following result is established.


\begin{lemma} \label{l:lsc} 
For any $\tau\in (0,T)$, the map $t \longmapsto \|V(t)\|_{\cL(Y_\tau)}$ is lower semi-continuous in
$[\tau,T]$.
\end{lemma}

\begin{proof} 
Let $D_\tau:= \cD(A)\times L^2(0,\tau;\cD(A))$.
Clearly $D_\tau$ is dense in $Y_\tau$.
Recalling that $\|V(t)\|_{\cL(Y_\tau)} \le \|P\|_Z +\|U\|_Z \le C$, and that $V(\cdot)X$ is differentiable at the point $t$ when $X\in D_\tau$, a straightforward density argument shows that $V(\cdot)$ is strongly continuous at the point $t$, i.e. 
\begin{equation*}
\lim_{s\to t} \|[V(s)-V(t)]X\|_{Y_\tau}=0 \qquad \forall X\in Y_\tau.
\end{equation*}
In particular,
\begin{equation*}
\|V(t)X\|_{Y_\tau} = \lim_{s\to t} \|V(s)X\|_{Y_\tau} \le \liminf_{s \to t} \|V(s)\|_{\cL(Y_\tau)} \|X\|_{Y_\tau}, \qquad \forall X\in Y_\tau  
\end{equation*}
and consequently
\begin{equation*}
\|V(t)\|_{\cL(Y_\tau)} \le \liminf_{s \to t} \|V(s)\|_{\cL(Y_\tau)}.
\end{equation*}
The proof is complete.
\end{proof}


\appendix

\section{Further analytical results} \label{a:appendix}
Here we collect a series of results which are instrumental for the proof of the feedback formula \eqref{e:feedback} and/or critically utilized in the derivation of the coupled system of (three) differential equations satisfied by the operators $P_i$, $i=0,1,2$, initially identified via the representation \eqref{e:riccati-ops}, subsequently shown to be equivalent to \eqref{e:riccati-ops_2}.

\subsection{Instrumental results, I} 
A first set of results specifically pertain to the operators
\begin{equation} \label{e:crucial-op}
\begin{split}
\mu(t)&= \int_0^t e^{(t-s)A}K(s)\, ds
\\
F(t,\tau)&=e^{(t-\tau)A}-\int_\tau^t R(t-s) e^{(s-\tau)A}\, ds
\\
G(t,\sigma,\tau)&= \mu(t-\sigma)-e^{(t-\tau)A}\mu(\tau-\sigma)
\\
M(t,\sigma,\tau)&=G(t,\sigma,\tau)-\int_\tau^t R(t-s) G(s,\sigma,\tau)\, ds
\end{split}
\end{equation}
introduced at the very outset in Section~\ref{s:prelimin} and whose respective definitions are recorded above for the reader's convenience; see \eqref{e:vari-ops} 
($R(\cdot)$ is the solution of the Volterra equation \eqref{e:resolvent-op}).

First, it is useful to list certain basic properties of the operators in 
\eqref{e:crucial-op} which follow immediately from the respective definitions.


\begin{lemma} \label{l:0}
The operators $\mu(t)$, $F(t,\tau)$, $G(t,\sigma,\tau)$ and $M(t,\sigma,\tau)$
recalled above in \eqref{e:crucial-op} satisfy
\begin{equation} \label{e:banali}
\begin{matrix}
& \mu(0)=0\,, \qquad F(\tau,\tau)=I\,;
\\[1mm]
& G(\tau,\sigma,\tau)=0\,, \quad G(t,\tau,\tau)=\mu(t-\tau)\,;
\\[1mm]
& M(\tau,\sigma,\tau)= 0\,, \quad M(t,\tau,\tau)=-R(t-\tau)\,.
\end{matrix}
\end{equation}
\end{lemma}

\begin{proof}
The conditions grouped collectively in \eqref{e:banali} are inferred via an easy verification.
The two conditions in the first row are immediate, on the basis of \eqref{e:crucial-op}.
Then, we note that $$G(\tau,\sigma,\tau)=\mu(\tau-\sigma)-\mu(\tau-\sigma)=0\,;$$
then $M(\tau,\sigma,\tau)= G(\tau,\sigma,\tau)=0$ as a consequence.
We have instead $$G(t,\tau,\tau)= \mu(t-\tau)-\mu(0)=\mu(t-\tau)\,,$$ 
and therefore,
\begin{equation*}
\begin{split}
M(t,\tau,\tau)&=G(t,\tau,\tau)-\int_\tau^t R(t-s) G(s,\tau,\tau)\, ds
= \mu(t-\tau) -\int_\tau^t R(t-s) \mu(s-\tau)\, ds
\\
&=\mu(t-\tau) -\int_0^{t-\tau} R(t-\tau-\lambda) \mu(\lambda)\, d\lambda
=-R(t-\tau)\,,
\end{split}
\end{equation*}
where in the last equality the Volterra equation \eqref{e:resolvent-op} satisfied by the resolvent operator $R$ is employed. 
\end{proof}

\smallskip
We move on exploring the differentiability of the operator $\mu$.


\begin{proposition} \label{p:1}
Let $\mu(\cdot)$ be the operator in \eqref{e:banali}.
If $x\in \cD(A)$, then $\mu(\cdot)x$ is differentiable, and
\begin{equation}\label{e:mu-diff}
\mu'(t)x=K(t)x+\mu(t)Ax\,, \qquad x\in \cD(A)
\end{equation}
holds true.

\end{proposition}

\begin{proof}
The proof is straighforward. 
With $t\ge 0$ and $h\ne 0$ ($h>0$, in the case $t=0$, also in the sequel) we compute
\begin{equation*}
\frac{\mu(t+h)-\mu(t)}{h}=\frac{1}{h} \int_t^{t+h} e^{(t+h-\lambda)A}K(\lambda)\,d\lambda
+ \int_0^t \frac{e^{(t+h-\lambda)A}-e^{(t-\lambda)A}}{h}K(\lambda)\,d\lambda\,,
\end{equation*}
which implies readily, for $x\in \cD(A)$, that there exists
\begin{equation*}
\begin{split}
\lim_{h\to 0}\frac{\mu(t+h)-\mu(t)}{h}x &=
K(t)x + \int_0^t Ae^{(t-\lambda)A}K(\lambda)x\,d\lambda
\\[1mm]
&= K(t)x + \int_0^t e^{(t-\lambda)A}K(\lambda)Ax\,d\lambda\,,
\end{split}
\end{equation*}
that is nothing but \eqref{e:mu-diff}.
\end{proof}

\begin{remark}
\begin{rm}
It is worth noting that the property that $R$ commutes with the operator $A$ -- and hence, with the semigroup $e^{tA}$ as well -- is essential for the proof of Proposition~\ref{p:1}. 
\end{rm}
\end{remark}


\begin{proposition} \label{p:2}
The resolvent operator $R(\cdot)$
is differentiable on $\cD(A)$, with
\begin{equation}\label{e:R-diff}
R'(t)x= -K(t)x + R(t)Ax+\int_0^t K(t-\sigma)R(\sigma)\,d\sigma\,,
\qquad x\in \cD(A)\,.
\end{equation}
\end{proposition}

\begin{proof}
Recall that $R(t)$ is expressed by the iterated formula \eqref{e:iterated}, and take
the convolution of $\mu$ and $R$, to find
\begin{equation*}
\mu\ast R= R\ast \mu = -\mu\ast\mu(t) - \mu\ast\mu\ast \mu(t)- \dots\,,
\end{equation*}

which once read from the left to the right yields $R(t)+\mu(t)=R\ast \mu$, that is

\begin{equation*}
R(t) 
= -\mu(t) +\int_0^t R(\sigma)\mu(t-\sigma)\,d\sigma\,.
\end{equation*}

The above applied to $x\in \cD(A)$ establishes that $R(t)x$ is differentiable in the
first place in view of Proposition~\ref{p:1}; then, a straightforward computation yields \eqref{e:R-diff} by virtue of \eqref{e:mu-diff}.

\end{proof}


\begin{proposition} \label{p:3}
If $x\in \cD(A)$, then there exist both $\partial_t F(t,\tau)x$ and $\partial_\tau F(t,\tau)x$, given by the following espressions:
\begin{equation}\label{e:F-derivatives}
\begin{split}
\partial_t F(t,\tau)x &=e^{(t-\tau)A}Ax-\int_\tau^t R'(t-s)e^{(s-\tau)A}x\,ds\,, 
\\[1mm]
\partial_\tau F(t,\tau)x &=-e^{(t-\tau)A}Ax+R(t-\tau)x-\int_\tau^t R(t-s)e^{(s-\tau)A}Ax\,ds
\\[1mm]
&=-F(t,\tau)Ax+R(t-\tau)x\,.
\end{split}
\end{equation}

\end{proposition}

\begin{proof}
The thesis follows readily from the definition of $F(t,\tau)$, in light of Proposition~\ref{p:2}.
\end{proof}


\begin{proposition} \label{p:5}
If $x\in \cD(A)$, then there exists
\begin{equation}\label{e:M-derivative}
\partial_\tau M(t,\sigma,\tau)x 
\end{equation} 

\end{proposition}

\begin{proof}
Let $x\in \cD(A)$ be given.
In view of the definition of $M(t,\sigma,\tau)x$ in \eqref{e:crucial-op},
we preliminarly examine $G(t,\sigma,\tau)x$.
Because of Proposition~\ref{p:1} and \eqref{e:mu-diff}, we see that there exists 
\begin{equation*}
\begin{split}
\partial_\tau G(t,\sigma,\tau)x&= 
-\frac{\partial}{\partial\tau}\big[e^{(t-\tau)A}\mu(\tau-\sigma)x\big]
\\
&= Ae^{(t-\tau)A}\mu(\tau-\sigma)x- e^{(t-\tau)A}\mu'(\tau-\sigma)x
\\
&\underset{\eqref{e:mu-diff}} = 
e^{(t-\tau)A}A\mu(\tau-\sigma)x-e^{(t-\tau)A}\big[K(\tau-\sigma)x+\mu(\tau-\sigma)Ax\big]
\\
&=-e^{(t-\tau)A}K(\tau-\sigma)x\,.
\end{split}
\end{equation*} 
Going back once more to the definitions in \eqref{e:crucial-op},
the above implies that -- when acting on $\cD(A)$ -- there exists 
\begin{equation*}
\begin{split}
\partial_\tau M(t,\sigma,\tau)
&=\partial_\tau G(t,\sigma,\tau)+R(t-\tau) \underbrace{G(\tau,\sigma,\tau)}_{\equiv 0} 
-\int_\tau^t R(t-s) \partial_\tau G(s,\sigma,\tau)\, ds
\\
&=-e^{(t-\tau)A}K(\tau-\sigma)+\int_\tau^t R(t-s) e^{(s-\tau)A}K(\tau-\sigma)\,ds
\\
&=-\Big[e^{(t-\tau)A}-\int_\tau^t R(t-s) e^{(s-\tau)A}\,ds\Big]\,K(\tau-\sigma)
= -F(t,\tau) K(\tau-\sigma)\,,
\end{split}
\end{equation*}
which confirms \eqref{e:M-derivative}.

\end{proof}


\subsection{Instrumental results, II} 
A second set of results pertains to the couples of operators 
$\Psi_1(t,\tau)$ and $\Psi_2(t,\sigma,\tau)$, $Z_1(t,\tau)$ and $Z_2(t,s,\tau)$,
identified earlier in the paper (see \eqref{e:def-psi12} and \eqref{e:def-Z12}).
Their respective expressions are collectively recalled here for the reader's convenience:
\begin{equation} \label{e:recall-4}
\begin{split}
\Psi_1(t,\tau) &= -\Big[\Lambda_\tau^{-1}L_\tau^*QF(\cdot,\tau)\Big](t)\,,
\\
\Psi_2(t,\sigma,\tau) &= -\Big[\Lambda_\tau^{-1}L_\tau^*QM(\cdot,\sigma,\tau)\Big](t)\,,
\\[1mm]
Z_1(t,\tau) &= F(t,\tau)+\int_\tau^t F(t,\sigma) B\Psi_1(\sigma,\tau)\,d\sigma\,,
\\
Z_2(t,s,\tau) &= M(t,s,\tau)+\int_\tau^t F(t,\sigma) B\Psi_2(\sigma,s,\tau)\,d\sigma\,.
\end{split}
\end{equation}

The representation formulas 
\begin{equation}\label{e:useful}
\begin{split}
Z_1(t,\tau) &= \big\{\big[I-L_\tau\Lambda_\tau^{-1}L_\tau^*Q\big]F(\cdot,\tau)\big\}(t)\,,
\\[1mm]
Z_2(t,s,\tau) &= \big\{\big[I-L_\tau\Lambda_\tau^{-1}L_\tau^*Q\big]M(\cdot,s,\tau)\big\}(t)\,.
\end{split}
\end{equation}
that follow inserting the expressions of $\Psi_1$ and $\Psi_2$ within the ones of
$Z_1$ and $Z_2$, respectively, are especially useful for the derivation of the
Riccati equation satisfied by $P_2$ (specifically in the proof of the statement S5. of 
Theorem~\ref{t:main}, part (ii), Lemma~\ref{l:4+6}).

\medskip

\subsubsection{Adjoint operators} 
 We begin by computing the adjoints of the operators listed in \eqref{e:recall-4}, as they occur in the very definition \eqref{e:riccati-ops} of the operators $P_i$, $i=0,1,2$.
The obtained expressions play a critical role in the derivation of the alternative (and neater) representations \eqref{e:riccati-ops_2} of the $P_i$, $i=0,1,2$; see Lemma~\ref{l:key-lemma}.

\begin{lemma}
Let $\Psi_1(t,\tau)$, $\Psi_2(t,\sigma,\tau)$, $Z_1(t,\tau)$ and $Z_2(t,s,\tau)$
be the operators recalled in \eqref{e:recall-4}.
The respective adjoint operators are acting as follows:
\begin{equation} \label{e:4-adjoints}
\begin{split}
\Psi_1(\cdot,\tau)^*g &= F(\cdot,\tau)^*Q [L_\tau \Lambda_\tau^{-1}g](\cdot)\,,
\\
\Psi_2(\cdot,\sigma,\tau)^*g &= M(\cdot,\sigma,\tau)^*Q[L_\tau\Lambda_\tau^{-1}g](\cdot)
\\
Z_1(\cdot,\tau)^*g &= F(\cdot,\tau)^*g(\cdot)+\Psi_1(\cdot,\tau)^*[L_\tau^*g](\cdot)
\\
Z_2(\cdot,s,\tau)^*g &= M(\cdot,s,\tau)^*g(\cdot) +\Psi_2(\cdot,s,\tau)^*[L_\tau^*g](\cdot)
\end{split}
\end{equation}
for every $g\in L^2(\tau,T;H)$.

\end{lemma}

\begin{proof}
(i) With $f,g\in L^2(\tau,T;H)$ we have
\begin{equation}
\begin{split}
&\int_\tau^T \big\langle \Psi_1(\sigma,\tau)f(\sigma),g(\sigma)\big\rangle_H\,d\sigma 
= - \int_\tau^T \big\langle \big[\Lambda_\tau^{-1}L_\tau^*QF(\cdot,\tau)\big](\sigma) f(\sigma),g(\sigma)\big\rangle_H\,d\sigma 
\\
& \myspace = - \int_\tau^T \big\langle F(\sigma,\tau)f(\sigma),Q\big[L_\tau \Lambda_\tau^{-1}g\big](\sigma)\big\rangle_H\,d\sigma
\\
& \myspace = - \int_\tau^T \big\langle f(\sigma),F(\sigma,\tau)^*Q\big[L_\tau \Lambda_\tau^{-1} g\big](\sigma)\big\rangle_H\,d\sigma\,,
\end{split}
\end{equation}
which establishes the first one of \eqref{e:4-adjoints}.

\smallskip
\noindent
(ii) We just repeat the preceding computation with $M(\cdot,\sigma,\tau)$ in place
of $F(\cdot,\tau)$, and the second one of \eqref{e:4-adjoints} follows.

\smallskip
\noindent
(iii) Once again, given $f,g\in L^2(0,T;H)$ we have 
\begin{equation}
\begin{split}
&\int_\tau^T \big\langle Z_1(q,\tau)f(q),g(q)\big\rangle_H\,d\sigma 
= - \int_\tau^T \Big\langle \Big[F(q,\tau)+\big[L_\tau\Psi_1(\cdot,\tau)\big](q)\Big]f(q),g(q)\Big\rangle_H\,dq 
\\
& \myspace =  \int_\tau^T \big\langle f(q),F(q,\tau)^*g(q)
+ \Psi_1(q,\tau)^*\big[L_\tau^*g\big](q)\big\rangle_H\,dq 
\end{split}
\end{equation}
which establishes the third one of the \eqref{e:4-adjoints}.

\smallskip
\noindent
(iv) The verification of the last one of the identities \eqref{e:4-adjoints} easily follows
as in (iii), just replacing $F(\cdot,\tau)$ and $\Psi_1(\cdot,\tau)$ by $M(\cdot,s,\tau)$ and $\Psi_2(\cdot,s,\tau)$, respectively.  
\end{proof}


\subsubsection{Differentiation of the operators}

\begin{proposition} \label{p:4}
Let $\Psi_1(t,\tau)$ as in \eqref{e:recall-4}.
If $x\in \cD(A)$, then the derivative $\partial_\tau\Psi_1(t,\tau)x$ exists, and it is given by
\begin{equation}\label{e:Psi1-derivative}
\partial_\tau\Psi_1(t,\tau)x
= -\big[\Lambda_\tau^{-1}L_\tau^*QF_\tau(\cdot,\tau)x\big](t)
+ \big[\Lambda_\tau^{-1}L_\tau^*QF(\cdot,\tau)B\Psi_1(\tau,\tau)x\big](t)\,.
\end{equation}

\end{proposition}

\begin{proof}
The proof's strategy is as follows: we consider the implicit representation of 
$\Psi_1(t,\tau)$, that is (since $\Lambda_\tau=I+L_\tau^*QL_\tau $)
\begin{equation*}
\Psi_1(t,\tau)+\big[L_\tau^*QL_\tau \Psi_1(\cdot,\tau)\big](t)
= - \big[L_\tau^*QF(\cdot,\tau)\big](t)\,,
\end{equation*}
which explicitly reads as
\begin{equation*}
\begin{split}
& \Psi_1(t,\tau)+\int_t^T B^*F(p,t)^*Q\int_\tau^p F(p,\sigma)B\Psi_1(\sigma,\tau)\,d\sigma\,dp
\\[1mm]
& \myspace = - \int_t^TB^*F(p,t)^*Q F(p,\tau)\,dp\,.
\end{split}
\end{equation*}
The Fubini-Tonelli Theorem yields
\begin{equation*} 
\begin{split}
& \Psi_1(t,\tau)+\Big[\int_\tau^t\int_t^T+\int_t^T \int_\sigma^T \Big] B^*F(p,t)^*Q F(p,\sigma)B\Psi_1(\sigma,\tau)\,dp\,d\sigma
\\[1mm]
& \myspace = - \int_t^TB^*F(p,t)^*Q F(p,\tau)\,dp\,,
\end{split}
\end{equation*}
that is
\begin{equation} \label{e:start-prop4}
\begin{split}
& \Psi_1(t,\tau)
+\int_\tau^t\Big[\int_t^TB^*F(p,t)^*Q F(p,\sigma)B\,dp\Big]\Psi_1(\sigma,\tau)\,d\sigma
\\
& \myspace
+ \int_t^T \Big[\int_\sigma^T B^*F(p,t)^*Q F(p,\sigma)B\Big]\,dp\Psi_1(\sigma,\tau)\,d\sigma
\\[1mm]
& \qquad = - \int_t^TB^*F(p,t)^*Q F(p,\tau)\,dp\,.
\end{split}
\end{equation}

Taking \eqref{e:start-prop4} as a starting point, we compute for $x\in \cD(A)$ the incremental ratio of $\Phi_1(t,\cdot)x$ (with $h\ne 0$) to find the identity
 
\begin{equation*} 
\begin{split}
& \frac{\Psi_1(t,\tau+h)x-\Psi_1(t,\tau)x}{h} 
\\[1mm]
& \myspace
-\frac{1}{h}\int_\tau^{\tau+h}\Big[\int_t^TB^*F(p,t)^*Q F(p,\sigma)B\,dp\Big]\Psi_1(\sigma,\tau)x\,d\sigma
\\[1mm]
& \myspace
+\int_\tau^t\Big[\int_t^TB^*F(p,t)^*Q F(p,\sigma)B\,dp\Big]
\frac{\Psi_1(\sigma,\tau+h)x-\Psi_1(\sigma,\tau)x}{h}\,d\sigma
\\[1mm]
& \myspace
+ \int_t^T \Big[\int_\sigma^T B^*F(p,t)^*Q F(p,\sigma)B\Big]\,dp
\frac{\Psi_1(\sigma,\tau+h)x-\Psi_1(\sigma,\tau)x}{h}\,d\sigma
\\[1mm]
& \qquad = - \int_t^TB^*F(p,t)^*Q \,\frac{F(p,\tau+h)x-F(p,\tau)x}{h}\,dp\,.
\end{split}
\end{equation*}

Thus we proceed as before, but somewhat in the reverse direction: we

\begin{itemize}

\item
use once again the Fubini-Tonelli Theorem, this time to merge the third and fourth summands
in the left hand side,

\item
recognize that the sum of the first, third and fourth terms is nothing but 
$I+L_\tau^*QL_\tau=:\Lambda_\tau$ applied to $[\Psi_1(\sigma,\tau+h)x-\Psi_1(\sigma,\tau)x]/h$,

\item
move the second summand from the left to the right hand side,
\end{itemize}

\smallskip
\noindent
to attain
\begin{equation*}
\begin{split} 
\frac{\Psi_1(t,\tau+h)x-\Psi_1(t,\tau)x}{h}&= 
\underbrace{\Lambda_\tau^{-1}
\Big[-L_\tau^*Q \frac{F(\cdot,\tau+h)x-F(\cdot,\tau)x}{h}\Big](t)}_{T_1(t)x}
\\[1mm]
& \qquad +\underbrace{\Lambda_\tau^{-1}\Big[
\frac{1}{h}\int_\tau^{\tau+h}\big[L_\tau^*QF(\cdot,\sigma)B\,\Psi_1(\sigma,\tau)x\big](t)\,d\sigma\Big]}_{T_2(t)x}\,.
\end{split}
\end{equation*}

Now the existence of the limit, as $h\to 0$, of the incremential ratio of $\Psi_1(t,\cdot)x$, along with the formula \eqref{e:Psi2-derivative}, follows observing that, since $x \in\cD(A)$, we have
\begin{equation*}
\begin{cases}
T_1(t)x \to -\big[\Lambda_\tau^{-1}L_\tau^*Q \partial_\tau F(\cdot,\tau)x\big](t)\,,
\; \text{as $h\to 0$} & \hspace{-1cm} \text{(in view of Proposition~\ref{p:3}),} 
\\[1mm]
T_2(t)x\to \big[\Lambda_\tau^{-1}L_\tau^*QF(\cdot,\tau)B\Psi_1(\tau,\tau)x\big](t)\,, \; \text{as $h\to 0$.} & 
\end{cases}
\end{equation*}

This establishes \eqref{e:Psi1-derivative}, thus concluding the proof of the proposition.

\end{proof}


\begin{proposition} \label{p:6}
Let $\Psi_2(t,\sigma,\tau)x$ as in \eqref{e:recall-4}. 
If $x\in \cD(A)$, then there exists $\partial_\tau\Psi_2(t,\sigma,\tau)x$ and it is given by
\begin{equation}\label{e:Psi2-derivative}
\partial_\tau\Psi_2(t,\sigma,\tau)x
= -\big[\Lambda_\tau^{-1}L_\tau^*QM_\tau(\cdot,\sigma,\tau)x\big](t)
+ \big[\Lambda_\tau^{-1}L_\tau^*QF(\cdot,\tau)B\Psi_2(\tau,\sigma,\tau)x\big](t)\,.
\end{equation}

\end{proposition}

\begin{proof}
The proof of the proposition can be carried out employing, {\em mutatis mutandis}, a similar path as the one utilized in the proof of Proposition~\ref{p:4}, with the implicit equation satisfied by $\Psi_2(t,\sigma,\tau)$ as a starting point, and the use
of Proposition~\ref{p:5} -- in place of Proposition~\ref{p:3} -- at the end.
The details are omitted.

\end{proof}

A direct consequence of Propositions~\ref{p:4} and \ref{p:6} are the following results, whose respective proofs are straightforward and hence are omitted.


\begin{proposition} \label{p:7}
Let $Z_1(t,\tau)$ as in \eqref{e:recall-4}.
Then $Z_1(\tau,\tau)=F(\tau,\tau)=I_H$.
If $x\in \cD(A)$, there exists $\partial_\tau Z_1(t,\tau)x$ and it is given by
\begin{equation}\label{e:Z1-derivative}
\partial_\tau Z_1(t,\tau)x=
\partial_\tau F(t,\tau)x-F(t,\tau)B\Psi_1(\tau,\tau)x+\big[L_\tau \partial_\tau \Psi_1(\cdot,\tau)x\big](t).
\end{equation}

\end{proposition}


\begin{proposition} \label{p:8}
Let $Z_2(t,s,\tau)$ as in \eqref{e:recall-4}.
Then, $Z_2(\tau,s,\tau)=M(\tau,s,\tau)=G(\tau,s,\tau) =0$.
If $x\in \cD(A)$, there exists $\partial_\tau Z_2(t,s,\tau)x$ and it is given by
\begin{equation}\label{e:Z2-derivative}
\partial_\tau Z_2(t,s,\tau)x=
\partial_\tau M(t,s,\tau)x-F(t,\tau)B\Psi_2(\tau,s,\tau)x+\big[L_\tau\partial_\tau\Psi_2(\cdot,s,\tau)x\big](t).
\end{equation}

\end{proposition}

\bigskip

{\small
\section*{Acknowledgements}
\noindent
The research of F.~Bucci has been performed in the framework of the MIUR-PRIN Grant 2020F3NCPX
``Mathematics for industry 4.0 (Math4I4)''.
Bucci's research was also supported by the Universit\`a degli Studi di Firenze
under the 2022 Project {\em Analisi e controllo di equazioni di evoluzione che descrivono propagazione ondosa o viscoelasticit\`a}, of which she was responsible.
F.~Bucci is a member of the Gruppo Nazionale per l'Analisi Mate\-ma\-tica, la Probabilit\`a  e le loro Applicazioni (GNAMPA) of the Istituto Nazionale di Alta Matematica (INdAM). 
%
}

\medskip


\end{document}